\documentclass{article}
\usepackage{graphicx} 
\usepackage{amsmath}
\usepackage{amssymb}
\usepackage{xcolor}
\usepackage{cancel}

\usepackage{enumitem}

\usepackage{amsthm}
\usepackage{MnSymbol}

\theoremstyle{plain}
\newtheorem{theorem}{Theorem}[section]
\newtheorem{lemma}[theorem]{Lemma}

\theoremstyle{definition}
\newtheorem{definition}[theorem]{Definition}
\newtheorem{example}[theorem]{Example}

\newtheorem{remark}[theorem]{Remark}

\DeclareMathOperator{\Config}{Config}
\DeclareMathOperator{\Inv}{Inv}

\DeclareMathOperator{\CNF}{CNF}

\DeclareMathOperator{\Table}{Table}
\DeclareMathOperator{\TAUT}{TAUT}
\DeclareMathOperator{\SAT}{SAT}

\DeclareMathOperator{\flag}{flag}
\DeclareMathOperator{\coord}{coord}

\newcommand{\pc}{\mathsf{pc}}
\newcommand{\pd}{\mathsf{pd}}

\newcommand{\tuple}{\mathbf}

\title{Lower Bounds on Inverse Cellular Automata via~Proof Complexity}
\author{Maryia Kapytka\thanks{Email: \texttt{mkapytka@protonmail.com}}\\
\small{Department of Algebra} \\
\small{Charles University}
}

\begin{document}

\maketitle

\begin{abstract}
We study the complexity of inverse cellular automata on configurations of bounded size. Deciding injectivity in this setting is co-NP-complete by a theorem of Durand~\cite{DURAND1994387}. We give a simpler proof of this theorem by a direct reduction from $\mathrm{UNSAT}$ to this problem, avoiding more complicated intermediate constructions.

We also show that one direction of the reduction can be formalized in the weak theory of bounded arithmetic $V^0$. 

Durand's coNP-completeness result allows one to view inverse cellular automata acting on bounded size configurations  as propositional proofs, cf. Cavagnetto~\cite{Cavagnetto,Cavagnetto2011}, and we prove lower bounds on their size. The proof uses known lower bounds for bounded-depth Frege systems together with the Paris--Wilkie translation of arithmetic proofs into propositional proofs, which allows us to transfer proof complexity lower bounds to our setting.
\end{abstract}

\section{Introduction}

Cellular automata are simple models of systems that change in steps, widely used in computer science and physics. They capture basic features of nature like local interactions, uniform rules, and parallel updates. A key concept is the Garden of Eden, a configuration that cannot arise from any previous state.  The Garden of Eden theorem characterizes when such configurations exist. A cellular automaton is surjective if and only if every configuration has a predecessor; that is, if and only if it has no Garden of Eden configuration.  This is further linked to local injectivity: two finite patterns are called twins if they can be swapped anywhere without affecting future configurations; the automaton is locally injective if it has no twins. The theorem states that surjectivity and local injectivity are equivalent. In two dimensions, Kari~\cite{kari} showed that checking injectivity or surjectivity is undecidable, meaning no algorithm can solve it in general. 

However, the situation changes when we restrict attention to bounded configurations. In this setting, the problem becomes decidable but remains computationally difficult. In particular, Durand~\cite{DURAND1994387} proved that deciding injectivity of two-dimensional cellular automata on bounded configurations is coNP-complete, implying that, unless $\mathrm{P} = \mathrm{NP}$, this problem is computationally hard. His proof uses a complex multi-step reduction from the nondeterministic halting problem via tiling systems.

In this paper, we give a more direct reduction. More precisely, we introduce a polynomial reduction of the coNP-complete set $\mathrm{UNSAT}$ (the set of unsatisfiable CNF formulas) to the set of injective two-dimensional cellular automata with von Neumann neighbourhood on configurations of bounded size. Moreover, we show that one direction of this reduction (namely, that if the formula is satisfiable, then the cellular automaton is not injective) can be formalized in the theory $V^0$, a weak theory of bounded arithmetic. Since $V^0$ is extremely weak, formalizing such a complexity-theoretic result within it is not straightforward. While several unprovability results are known for $V^0$, few non-trivial theorems of this kind have been shown to be provable inside it.

Durand \cite[Sec.~4]{DURAND1994387} posed the question of whether the inverse automaton to a cellular automaton operating on bounded configurations, whose transition function is computed by a simple algorithm, is also computed by a simple algorithm. Cavagnetto \cite[Sec.~3.3]{Cavagnetto} showed, using the computational complexity hypothesis that a one-way function exists, that this is not the case if by an algorithm one means a Boolean circuit. In addition, he showed that Durand’s coNP-completeness results still hold in such a setup. The question of the size of inverse cellular automata for the original representation of the transition function by its table was left open in both \cite{DURAND1994387,Cavagnetto}. We shall use proof complexity to solve this problem.

The motivation for our approach comes from the PhD thesis of Stefano Cavagnetto~\cite{Cavagnetto}, which was later published in~\cite{Cavagnetto2011}. In that work, inverse cellular automata were viewed as propositional proofs in the sense of Cook--Reckhow~\cite{cook1979relative}. We follow this idea.

We show that, when we restrict ourselves to configurations of bounded size, cellular automata in our construction have the property that inverse automata to them cannot be in general of subexponential size. Indeed, their size is exponential, and, in particular, their neighborhood must grow polynomially.

To prove this, we consider the unsatisfiable CNF formula 
\(
\varphi = \lnot \mathrm{ontoPHP}_k
\), which expresses that there is a bijection between $k+1$ and $k$. This formula has size polynomial in $k$ but requires
exponential-size bounded-depth Frege refutations by known strengthenings of a theorem of Ajtai~\cite{Ajtai1988}. Then we use the Paris--Wilkie translation~\cite{ParisWilkie1985} to simulate our inverse cellular automaton proof system by a bounded-depth Frege system $F_d$, thereby transferring this lower bound to our setting.

For background on cellular automata, see for example~\cite{kari}, and for background on bounded arithmetic and proof complexity, see~\cite{krajicek2019}.

\section{Preliminaries}

Throughout the paper we use the notation $[n] := \{0,1,\dots,n-1\}$.

\subsection{Cellular automata}\label{subsec:cellular_automata}

We briefly recall the basic notions of cellular automata that will be needed later.
The definitions are standard and follow~\cite{DURAND1994387}, but we include them here
to fix notation and terminology. We also record a simple bound on the size
needed to encode a two-dimensional cellular automaton with a fixed neighbourhood.

\begin{definition}[Cellular automaton]~\label{def:CA}

A cellular automaton is formally defined as a quadruplet
$(d, S, N, f)$:
\begin{itemize}
    \item The integer $d \geq 1$ is the \emph{dimension} of the space on which the cellular automaton operates.
    \item $S \not = \emptyset$ is a finite set called the set of \emph{states}.
    \item The \emph{neighbourhood} $N$ is a $v$-tuple of distinct vectors of $\mathbb{Z}^d$. For $N = (x_1, . . ., x_v)$, vectors $x_i$ are the relative positions of the neighbour cells with respect to a given center cell. The states of these neighbours are used to compute the new state of the center cell.
    \item The \emph{local function} of the cellular automaton $f:
S^v \to S$ gives the local transition rule.
\end{itemize}
\end{definition}

A \emph{configuration} is a map 
$C : \mathbb{Z}^d \to S$.
For a configuration $C$ and a position $y \in \mathbb{Z}^d$, 
the value $C(y)$ is called the \emph{state of the cell at position $y$}.  
We refer to the position $y$ itself as a \emph{cell}.

If $C : \mathbb{Z}^d \to S$ is a configuration, we identify $C$
with its graph
\[
\{\, ((i,j), C(i,j)) \mid (i,j) \in \mathbb{Z}^d \,\}.
\]

All automata considered in this paper are two-dimensional, i.e. $d=2$.

Sometimes, a state $q$ for which $f(q, q, \dots, q) = q$ is distinguished
in $S$ and is called
a \emph{quiescent} state. A \emph{finite configuration} is a configuration which has only a finite number of cells in a non-quiescent state. If there exist two integers $i$ and $j$ such that all non-quiescent cells of the
configuration are located inside a rectangle with coordinates in $[i] \times [j]$, then we say that the
finite configuration
is \emph{inside $[i] \times [j]$}. We shall be interested in how automata operate on configurations inside a fixed rectangle.

\begin{lemma}(\cite{DURAND1994387})\label{lm:CA_bound}
     For fixed $N$ (of size $v$) the size of a binary string sufficient to code a cellular automaton of dimension $d=2$ is bounded above by $s^v\log s$, where $s$ is the size of $S$.
\end{lemma}
\begin{proof}
We have $d=2$ and a fixed neighbourhood N
(lexicographically ordered), so the local function can be represented by a table that has $s^v$ rows. Any value of the function is encoded by $\log s$ bits, so $s^v \log s$ bits in total suffice to encode the cellular automaton.
\end{proof}

\subsection{Propositional proof complexity}

Propositional proof complexity studies the length and structure of proofs of propositional tautologies in different proof systems. In this subsection we recall the notion of
a propositional proof system and introduce Frege and bounded-depth Frege systems. These concepts will later be linked to bounded arithmetic via the Paris--Wilkie translation.

The language we use in this subsection is the standard DeMorgan propositional language, consisting of propositional variables $x, y, z, \dots$, the logical constants $\top$ and $\bot$, the connectives $\neg$, $\wedge$, and $\vee$, and parentheses.

\subsubsection{Proof systems}\label{proof_systems}

Let $\TAUT$ denote the set of propositional tautologies and let $\SAT$ denote the set of satisfiable propositional formulas. We identify formulas with binary strings that encode them.

\begin{definition}[Propositional proof system {\cite[Definition~1.3]{cook1979relative}}]\label{def:proof_system}



A \emph{relational propositional proof system} is a binary relation 
\(Q \subseteq \{0,1\}^* \times \{0,1\}^*\), where $\{0,1\}^*$ denotes the set of all finite binary strings, such that:
\begin{itemize}
    \item \(Q\) is p-time decidable;
    \item for any \(w, \alpha\), if \(Q(w,\alpha)\) holds then \(\alpha \in \mathrm{TAUT}\);
    \item for any \(\alpha \in \mathrm{TAUT}\), there is a \(w \in \{0,1\}^*\) such that \(Q(w,\alpha)\) holds.
\end{itemize}

The second condition gives the soundness of \(Q\) and the third its completeness.
\end{definition} 

We will often shorten the terminology \emph{propositional proof system} to \emph{proof system}.

A \emph{refutation} of a formula $\varphi$ in $Q$ is a proof 
$w \in \{0,1\}^*$ such that 
\[
Q(w,\neg \varphi)
\]
holds. 
Equivalently, a refutation of $\varphi$ is a proof of the tautology 
$\neg \varphi$.

The size of a proof $w$ (respectively, a formula $\varphi$) is the number of symbols it contains.

To compare the strength of different propositional proof systems, we use the notion
of polynomial simulation.

\begin{definition}[p-simulation]
Let $P$ and $Q$ be two propositional proof systems.  
A p-time function $f : \{0,1\}^* \to \{0,1\}^*$ is a \emph{p-simulation} of $Q$ by $P$ if and only if for all strings $w, \alpha$,
\[
    Q(w,\alpha) \;\rightarrow\; P(f(w,\alpha), \alpha).
\]
If such $f$ exists, we say that $P$ \emph{p-simulates} $Q$.
\end{definition}

We now recall the definition of a Frege proof system, a standard example of a propositional proof system.

A formula $B$ is \emph{logically implied} by a formula $A$, written $A \models B$, if every assignment that makes $A$ true also makes $B$ true. 
If $A = A_1 \wedge A_2 \wedge \dots \wedge A_k$ with $k \geq 2$, we write $A_1, A_2, \dots, A_k \models B$ instead of $A_1 \wedge A_2 \wedge \dots \wedge A_k \models B$.  If $k=0$, we write $\models B$ and this means that $B$ is a tautology.

Let $k \ge 0$.  
A \emph{$k$-ary Frege rule} is a $(k+1)$-tuple of formulas $(A_0,A_1,\dots,A_k)$, written
\[
\frac{A_1,\dots,A_k}{A_0},
\]
such that
\[
A_1,\dots,A_k \models A_0.
\]

A $0$-ary Frege rule is called \emph{an axiom}.

\medskip

A standard example of a Frege rule is \emph{modus ponens}:
\[
\frac{p,\;\neg p \vee q}{q}.
\]

Let $F$ be a finite set of Frege rules.
An \emph{$F$-proof} of a formula $C$ from formulas $B_1,\dots,B_\ell$ is a finite
sequence $(D_1,\dots,D_m)$ such that:
\begin{enumerate}
  \item $D_m = C$;
  \item for each $i \le m$, either
  \begin{enumerate}
    \item $D_i = \sigma(A_0)$ for some rule
    \(
      \frac{A_1,\dots,A_k}{A_0} \in F
    \), indices $j_1,\dots,j_k < i$, and substitution $\sigma$ with
    $\sigma(A_r)=D_{j_r}$ for all $r \le k$, or
    \item $D_i \in \{B_1,\dots,B_\ell\}$.
  \end{enumerate}
\end{enumerate}

We write
\[
B_1,\dots,B_l \vdash_F C
\]
if there exists an $F$-proof of $C$ from $B_1,\dots,B_l$. When $F$ is clear from the
context, we omit the subscript.

\medskip

A \emph{Frege proof system} is a finite set $F$ of Frege rules that is
\emph{sound} and \emph{(implicationally) complete}, meaning that for all formulas
$B_1,\dots,B_\ell,C$,
\[
B_1,\dots,B_\ell \vdash_F C \;\Rightarrow\; B_1,\dots,B_\ell \models C
\quad\text{(soundness)}
\]
and
\[
B_1,\dots,B_\ell \models C \;\Rightarrow\; B_1,\dots,B_\ell \vdash_F C
\quad\text{(completeness)}.
\]

A key measure of the logical complexity of propositional formulas is their depth.
Bounded-depth Frege systems are defined by requiring that proofs use only formulas of a depth bounded by an independent constant. We therefore begin by defining the depth of a formula.

\begin{definition}[Depth of a formula]
    The notion of the \emph{depth} of $A$, denoted by $\operatorname{dp}(A)$, is defined by the following inductive definition:
\begin{itemize}
    \item propositional variables and the constants have depth $0$;
    \item $\operatorname{dp}(\lnot B) = \operatorname{dp}(B)$ if $B$ starts with $\lnot$ and $\operatorname{dp}(\lnot B) = 1 + \operatorname{dp}(B)$ otherwise;
    \item if $B(x_1,\ldots,x_t)$ is built from propositional variables and disjunctions only and none of the formulas
    $C_1,\ldots,C_t$ starts with a disjunction, then
    \[
        \operatorname{dp}\!\bigl(B(C_1,\ldots,C_t)\bigr)
        = 1 + \max_i \operatorname{dp}(C_i);
    \]
    \item a condition holds for conjunction analogous to the previous item.
\end{itemize}
\end{definition}

For a Frege system $F$ and $d \ge 0$ we write $F_d$ for the restriction of $F$ in which proofs are allowed to use only formulas of depth at most $d$. Frege systems of this form are called \emph{bounded-depth} (or \emph{constant-depth}) Frege systems. Note that these systems are not complete for all tautologies (as these have unbounded depth), but it is well-known that they are complete for formulas of a bounded depth.

We now give a basic and well-known example from propositional proof complexity.
The pigeonhole principle is a standard source of hard formulas and is often used
to show lower bounds for proof systems, in particular for bounded-depth Frege systems.

\begin{example}[Pigeonhole principle]\label{ex:pigeonhole}
 Consider the following CNF formula $\lnot onto\mathrm{PHP}_n$, consisting of clauses
over propositional variables $p_{ij}$, where \linebreak $i~\in~ [n+1]$ and $j ~\in~[n]$:
\begin{enumerate}
  \item $\bigvee_j p_{ij}$, one clause for each $i$;
  \item $\lnot p_{i_1 j} \lor \lnot p_{i_2 j}$, one clause for each triple
        $i_1 < i_2$ and $j$;
  \item $\lnot p_{i j_1} \lor \lnot p_{i j_2}$, one clause for each triple
        $i$ and $j_1 < j_2$;
  \item $\bigvee_i p_{ij}$, one clause for each $j$.
\end{enumerate}

The variable $p_{ij}$ is intended to express that pigeon $i$ is assigned to hole $j$.
The clauses in (1) and (4) assert that every pigeon is assigned to some hole and that
every hole contains at least one pigeon, respectively. Clauses in (2) and (3) enforce
that no two pigeons are assigned to the same hole and that no pigeon is assigned to
more than one hole.

Clearly, $\neg onto\operatorname{PHP}_n$ is unsatisfiable, since the number of pigeons exceeds the number of holes.  In fact, the formula remains unsatisfiable even if only clauses
(1) and (2) are retained. This weaker principle is known as the \emph{weak pigeonhole
principle}, denoted $\lnot \mathrm{PHP}_n$.

We introduce this example because of the following well-known result due to Ajtai,
which establishes super-polynomial lower bounds for bounded-depth Frege proofs.

\begin{theorem}[Ajtai’s theorem]\cite{Ajtai1988}\label{thm:ajtai}
Let $d \ge 2$ and let $0 < \delta < 5^{-d}$. Then for all sufficiently large $n$ the following holds.

Every $F_d$-refutation of $\lnot \mathrm{ontoPHP}_n$ (that is, every $F_d$-proof of $\bot$ from
\linebreak$\lnot \mathrm{ontoPHP}_n$) contains at least $2^{n^\delta}$ pairwise different subformulas.

In particular, any such refutation has size at least $2^{n^\delta}$.
\end{theorem}

\end{example}

\subsection{Bounded arithmetic}

\subsubsection{$I\Delta_0$ and $I\Delta_0(R)$}\label{I_delta_zero}

Throughout this subsection we work in the language of Peano arithmetic
\[
L_{PA}=\{0,1,+,\cdot,\le\}.
\]
A \emph{bounded quantifier} is a quantifier of the form
\[
\exists y \le t(\tuple x)\, A(\tuple x,y)
\quad\text{or}\quad
\forall y \le t(\tuple x)\, A(\tuple x,y),
\]
where $t(\tuple x)$ is an $L_{PA}$-term. It is an abbreviation for the formulas
\[
\exists y\,(y \le t(\tuple x)\wedge A(\tuple x,y))
\quad\text{and}\quad
\forall y\,(y \le t(\tuple x)\to A(\tuple x,y)),
\]
respectively.

The class of \emph{bounded formulas} $\Delta_0$ is then defined as the smallest class of
$L_{PA}$-formulas that contains all quantifier-free formulas and is closed under
DeMorgan connectives and bounded quantifiers.

\begin{definition}[$I\Delta_0$]\label{def:i_delta_zero}
The theory $I\Delta_0$ is the theory in $L_{PA}$ axiomatized by Robinson's arithmetic $Q$:
\[
\begin{array}{lll}
\mathbf{B1}. & x + 1 \neq 0 \\[2pt]
\mathbf{B2}. & x + 1 = y + 1 \rightarrow x = y \\[2pt]
\mathbf{B3}. & x + 0 = x \\[2pt]
\mathbf{B4}. & x + (y+1) = (x+y) + 1 \\[2pt]
\mathbf{B5}. & x \cdot 0 = 0 \\[2pt]
\mathbf{B6}. & x \cdot (y+1) = (x \cdot y) + x \\[6pt]
\end{array}
\qquad
\begin{array}{lll}
\mathbf{B7}. & (x \le y \land y \le x) \rightarrow x = y \\[2pt]
\mathbf{B8}. & x \le x + y \\[2pt]
\mathbf{B9}. & 0 \le x \\[2pt]
\mathbf{B10}. & x \le y \;\lor\; y \le x \\[2pt]
\mathbf{B11}. & x \le y \;\leftrightarrow\; x < y + 1 \\[2pt]
\mathbf{B12}. & x \neq 0 \;\rightarrow\; \exists y \le x \,(y+1 = x) \\[6pt]
\end{array}
\]

and by the axiom scheme of induction $IND$, where $A$ is a $\Delta_0$-formula:

\[
\neg A(0) \vee (\exists y<x \wedge A(y) \wedge \neg A(y+1)) \vee A(x).
\]

(We implicitly assume that all free variables are universally quantiﬁed.)
\end{definition}

Extend the language $L_{PA}$ by a new binary relation symbol $R(x,y)$, and let
$\Delta_0(R)$ be the class of bounded formulas in the extended language.
Define $I\Delta_0(R)$ as Robinson’s arithmetic $Q$ together with the induction
scheme $IND$ restricted to $\Delta_0(R)$-formulas.

In our application we will need to work with several additional relation symbols and constants.   
Let $L_{PA}(\alpha)$ denote the language obtained from $L_{PA}$ by adding an arbitrary countable collection of new relation symbols and constants, where $\alpha$ serves as a placeholder for this extension. We define the theory $I\Delta_0(\alpha)$ to consist of Robinson’s arithmetic $Q$ together with the induction scheme $IND$ for all $\Delta_0(\alpha)$-formulas.

\subsubsection{$V^0$}

The following subsection is based on~\cite[Section~5.1]{Cook_Nguyen_2010}. We introduce theory $V^0$ because it provides a convenient framework for formalizing arguments about the existence of sets. An important property of $V^0$ is that it is conservative over $I\Delta_0(\alpha)$. This means that every sentence in the language $\Delta_0(\alpha)$ that can be proved in $V^0$ can already be proved in $I\Delta_0(\alpha)$. 

This conservativity allows us to carry out certain arguments within $V^0$, and then transfer the resulting statements back to $I\Delta_0(\alpha)$. This will make several proofs in later sections simpler.

Two-sorted first-order logic is an extension of the usual single-sorted first-order
logic. Here there are two kinds of variables:
\begin{itemize}
    \item \emph{Number variables} $x, y, z, \dots$ range over natural numbers.
\item  \emph{Set variables} $X, Y, Z, \dots$ range over finite subsets of natural numbers.
\end{itemize}

The language of $V^0$ is

\[
\mathcal{L}^2_A = \{\, 0,1,+,\cdot,\mid\mid \ ;\ =_1, =_2, \le, \in \, \}.
\]

Here:
\begin{itemize}
    \item $0,1,+,\cdot,=_1$ come from $\mathcal{L}_{PA}$; $=_1$ corresponds to $=$ of $\mathcal{L}_{PA}$
    \item $\lvert X \rvert$ (the “length of $X$”) is a number-valued function and is intended to denote
the least upper bound of the set $X$ (essentially the largest element plus one)
\item $\in$ is the usual membership predicate
\item $=_2$ is the equality for sets
\end{itemize}
 We will write $=$ for both $=_1$ and $=_2$, the meaning
will be clear from the context.

We will use the abbreviation
$X(t)$ to denote $t \in X$. Thus we may think of $X(i)$ as the
$i$-th bit of the binary string $X$.

The theory $2$-BASIC extends basic arithmetic to sets. 

\begin{definition}[$2$-BASIC]
The theory $2$-BASIC consists of the axioms \linebreak $\mathbf{B1\text{--}B12}$ from
Definition~\ref{def:i_delta_zero}, together with the following axioms
for bounded sets:
\[
\begin{array}{lll}
\mathbf{L1}. & X(y) \;\rightarrow\; y < |X|, \\[4pt]
\mathbf{L2}. & y + 1 = |X| \;\rightarrow\; X(y), \\[4pt]
\mathbf{SE}. &
\Bigl(\,|X| = |Y| \;\land\; \forall i < |X|\bigl(X(i) \leftrightarrow Y(i)\bigr)\Bigr)
\;\rightarrow\; X = Y .
\end{array}
\]
\end{definition}

These axioms formalize basic properties of finite sets: the elements of a set are within its length, the last element determines the size of the string the set represents, and sets are equal if they have the same length and elements.

\begin{definition}[$\Sigma_0^B$-formulas]
$\Sigma_0^B$ is the set of $\mathcal{L}_A^2$-formulas whose
only quantifiers are bounded number quantifiers (there can be free string variables).

\end{definition}

Similarly as bounded (number) quantifiers, \emph{bounded set quantifiers}  are quantifiers of the form
\[
\exists X \le t(\tuple x, \tuple Y)\, A(\tuple x, \tuple Y, X)
\quad\text{and}\quad
\forall X \le t(\tuple x, \tuple Y)\, A(\tuple x, \tuple Y, X),
\]
where $t(\tuple x, \tuple Y)$ is an $\mathcal{L}_A^2$-term. It is an abbreviation for the formulas
\[
\exists X\,\bigl(|X| \le t(\tuple x, \tuple Y)\wedge A(\tuple x, \tuple Y, X)\bigr)
\quad\text{and}\quad
\forall X\,\bigl(|X| \le t(\tuple x, \tuple Y)\to A(\tuple x, \tuple Y, X)\bigr),
\]
respectively.

\begin{definition}[Bounded Comprehension Axiom]
The bounded comprehension axiom scheme for $\Sigma^B_0$, denoted by $\Sigma^B_0\text{-COMP}$, is the set of all formulas
\[
\exists X \le y \; \forall z < y \bigl( X(z) \leftrightarrow \varphi(z) \bigr),
\]
where $\varphi(z)$ is any formula in $\Sigma^B_0$, and $X$ does not occur free in $\varphi(z)$.

In the above definition $\varphi(z)$ may have free variables of both sorts, in addition to $z$.
\end{definition}

\begin{definition}[$V^0$]
The theory $V^0$ has the language $\mathcal{L}^2_A$ and is axiomatized by $2$-BASIC and $\Sigma^B_0\text{-COMP}$.
\end{definition}

Recall that for theories $T \subseteq S$, we say that $S$ is a \emph{conservative extension} of $T$ if, for every formula $\varphi$ in the language of $T$, the implication
\[
S \vdash \varphi \;\Rightarrow\; T \vdash \varphi
\]
holds.

\begin{theorem}\label{thm:conservative_extension}
The theory $V^0$ is a conservative extension of $I\Delta_0(\alpha)$.
\end{theorem}

In other words, extending $I\Delta_0(\alpha)$ by set variables and the bounded comprehension scheme does not yield new theorems that are purely first-order, i.e. do not use second-order quantifiers. This fact will be used repeatedly: it allows us to apply standard results about $\Delta_0(\alpha)$-formulas proved in $I\Delta_0(\alpha)$ directly within $V^0$, thereby simplifying many of our arguments.

\subsubsection{Coding of finite relations and functions}\label{coding_relations_functions}
Note that $V^0$ enables us to talk about finite subsets of natural numbers. In order to be able to discuss finite relations and bounded functions, we introduce their coding in $V^0$.

In order to define finite subsets of $\mathbb{N}^k$ for $k\geq 2$, we use the usual \emph{pairing function} 
\[
\langle x, y \rangle \;:=\; \frac{(x+y)(x+y+1)}{2} + x, 
\]
which is a bijection between $\mathbb{N} \times \mathbb{N}$ and $\mathbb{N}$, with the
projections simply definable with a bounded graph:
\[
p_1(z) = x 
\quad\text{if and only if}\quad 
\exists y \le z \; \langle x, y \rangle = z,
\]

\noindent and similarly for $p_2(z) = y$. Iterating this, we can encode $k$-tuples of numbers
for any fixed $k \ge 2$. Using this coding, a finite relation $R \subseteq \mathbb{N}^k$ is represented as a
finite set of natural numbers coding the tuples in $R$.

Using the pairing function we can also code finite sequences of natural
numbers. A finite sequence
\[
(a_0,a_1,\dots,a_{k-1})
\]
is represented by the finite set
\[
S = \{\langle 0,a_0\rangle,\langle 1,a_1\rangle,\dots,\langle k-1,a_{k-1}\rangle\}.
\]

Thus finite sequences can be treated as finite relations
\(S \subseteq \mathbb{N}^2\) that satisfy the usual uniqueness condition
\[
\forall i\,\forall y_1,y_2\;
\bigl((\langle i,y_1\rangle \in S \wedge \langle i,y_2\rangle \in S)
\rightarrow y_1=y_2\bigr).
\]

Functions are represented by their graphs.
The \emph{graph of a function \(G\)} is a finite relation
\(G \subseteq M \times N\), where \(M \subseteq \mathbb{N}^k\) and
\(N \subseteq \mathbb{N}^l\) are finite sets for some \(k,l \in \mathbb{N}\).
The relation \(G\) is required to satisfy the usual conditions of
\emph{uniqueness} and \emph{totality}:
\[
\begin{aligned}
\text{(Uniqueness)}\qquad 
&\forall x \in M\, \forall y_1, y_2 \in N\;
   \bigl((x,y_1) \in G \wedge (x,y_2) \in G \;\rightarrow\; y_1 = y_2\bigr), \\[6pt]
\text{(Totality)}\qquad 
&\forall x \in M\, \exists y \in N\; (x,y) \in G.
\end{aligned}
\]

\subsection{Propositional translation of bounded arithmetic}

In this subsection we explain how statements of bounded arithmetic are translated into
propositional formulas and how this translation connects provability in $V^0$ with
propositional proof complexity. First, we recall the Paris--Wilkie translation, which
associates to each bounded arithmetic formula a propositional
formula of polynomial size and bounded depth.

\subsubsection{Paris-Wilkie translation}

The translation applies to all bounded formulas $\Sigma^B_0$ with any number of free set variables. To keep the notation simple we shall consider just the example when the formula contains an unspecified binary relation $R$ as that is the case needed for the translation of the pigeonhole principle. 

The Paris–Wilkie translation \cite[Section~8.2]{krajicek2019} assigns to any $\Delta_0(R)$-formula 
$A(x_1, \ldots, x_k)$ and any $n_1, \ldots, n_k \ge 0$ a DeMorgan propositional formula
\[
\langle A(\tuple{x}) \rangle_{n_1, \ldots, n_k}
\]
defined by induction on the logical complexity of $A$ as follows.

\begin{enumerate}
    \item If $B$ is one of the atomic formulas $t(\tuple{x}) = s(\tuple{x})$ or $t(\tuple{x}) \le s(\tuple{x})$, with $t$ and $s$ terms, then
    \[
    \langle B \rangle_{n_1, \ldots, n_k} :=
    \begin{cases}
        1 & \text{if } B(n_1, \ldots, n_k) \text{ is true,} \\
        0 & \text{otherwise.}
    \end{cases}
    \]
    
    \item If $B$ is the atomic formula $R(t(\tuple{x}), s(\tuple{x}))$ and $i$ and $j$ are the values of the terms $t(\tuple{x})$ and $s(\tuple{x})$ for $\tuple{x} := \tuple{n}$, respectively, then put
    \[
    \langle B \rangle_{n_1, \ldots, n_k} := r_{ij},
    \]
    where $r_{ij}$ are propositional variables.
    
    \item $\langle \ldots \rangle$ commutes with $\neg, \lor, \land.$
    
    \item If $A(\tuple{x}) = \exists y \le t(\tuple{x}) B(\tuple{x}, y)$ then
    \[
    \langle A(\tuple{x}) \rangle_{n_1, \ldots, n_k} := 
    \bigvee_{m \le t(\tuple{n})}
    \langle B(\tuple{x}, y) \rangle_{n_1, \ldots, n_k, m}.
    \]
    
    \item If $A(\tuple{x}) = \forall y \le t(\tuple{x}) B(\tuple{x}, y)$ then
    \[
    \langle A(\tuple{x}) \rangle_{n_1, \ldots, n_k} := 
    \bigwedge_{m \le t(\tuple{n})}
    \langle B(\tuple{x}, y) \rangle_{n_1, \ldots, n_k, m}.
    \]
\end{enumerate}

\begin{lemma}
Let $A(\tuple{x})$ be a $\Delta_0(R)$-formula. Then there are $c, d \ge 1$ such that, for all $n_1, \ldots, n_k,$
\begin{itemize}
    \item $|\langle A(\tuple{x}) \rangle_{n_1, \ldots, n_k}| \le (n_1 + \cdots + n_k + 2)^c,$
    \item $\mathrm{dp}(\langle A(\tuple{x}) \rangle_{n_1, \ldots, n_k}) \le d.$
\end{itemize}

\noindent Moreover, $A(n_1, \ldots, n_k)$ is true for all interpretations of $R$ if and only if 
\\$\langle A(\tuple{x}) \rangle_{n_1, \ldots, n_k}\in \mathrm{TAUT}.$
\end{lemma}

The key result about the Paris–Wilkie translation is that translations \linebreak $\langle A(\tuple{x}) \rangle_{n_1, \ldots, n_k}$
of bounded formulas $A(\tuple{x})$ whose universal closure is provable in $I\Delta_0(\alpha)$ have constant depth and poly-size Frege proofs. We shall formulate the theorem for $V^0$; this follows from Paris-Wilkie theorem~\cite{ParisWilkie1985} 
using the fact that $V^0$ is conservative over $I\Delta_0(\alpha)$.

\begin{theorem}\label{thm:paris_wilkie_v0}
Let $F$ be a Frege system in the DeMorgan language. Let $A(\tuple{x})$ be a $\Sigma^B_0$-formula and assume that $V^0$ proves $\forall \tuple{x}\, A(\tuple{x})$. Then there are $c, d \ge 1$ such that for each $k$-tuple $n_1, \ldots, n_k$ there is an $F$-proof $\pi_{n_1, \ldots, n_k}$ of $\langle A(\tuple{x}) \rangle_{n_1, \ldots, n_k}$ such that
\[
\mathrm{dp}(\pi_{n_1, \ldots, n_k}) \le d 
\quad \text{and} \quad
|\pi_{n_1, \ldots, n_k}| \le (n_1 + \cdots + n_k + 2)^c.
\]
\end{theorem}

\subsubsection{Coding of CNFs}\label{subsub:coding_cnfs}
Assume $\varphi$ is a CNF formula with $n$ clauses $C_1, C_2, \dots, C_{n}$ and $m$ propositional variables $x_1, x_2, \dots, x_{m}$. We encode it as a ternary relation $R_\varphi \subseteq \{1, 2, \dots, n\} \times \{1, 2, \dots, m\} \times \{-1, 0, 1\}$ defined as follows:
\begin{align*}
(i,j,1) \in R_\varphi &\;\;\Longleftrightarrow\;\; x_j \in C_i, \\
(i,j,-1) \in R_\varphi &\;\;\Longleftrightarrow\;\; \neg x_j \in C_i, \\
(i,j,0) \in R_\varphi &\;\;\Longleftrightarrow\;\; x_j \notin C_i \ \wedge\ \neg x_j \notin C_i .
\end{align*}

We will abuse the notation and use $\varphi$ to denote both the formula $\varphi$ and the corresponding relational encoding $R_\varphi$.

We say that a ternary relation $R \subseteq \mathbb{N}^3$
\emph{encodes a CNF formula with $n$ clauses and $m$ propositional variables}
if the predicate $\mathrm{CNF}(R,n,m)$ holds, where
\begin{align*}
\mathrm{CNF}(n,m,R) \;:=\; &
\Bigl(
R \subseteq \{1, 2, \dots, n\} \times \{1, 2, \dots, m\}\times\{-1,0,1\} \\ &
\;\wedge\;
R \text{ is a graph of a function }  \\&\{1, 2, \dots, n\} \times \{1, 2, \dots, m\} \rightarrow \{-1, 0, 1\}
\Bigr).
\end{align*}

An assignment $\tuple a = (a_1, a_2, \dots, a_m) \in \{0,1\}^m$ \emph{satisfies} $\varphi$, iff $\mathrm{Sat}(n,m, \tuple a, \varphi)$ holds, where 
\begin{align*}
\mathrm{Sat}(n,m, \tuple a, \varphi) \; := &\;   \forall i \in \{1, 2, \dots, n\} \;
\exists j \in \{1, 2, \dots, m\}\;
\exists s \in \{-1,1\}\; \\&
\bigl(R_\varphi(i,j,s) \wedge
\big[(s=1 \wedge a_j=1)\;\vee\;(s=-1 \wedge a_j=0)\bigr]\bigr).
\end{align*}

\subsubsection{Reflection principle}\label{reflection_principle}

In the following subsection we recall the reflection principle for an arbitrary
propositional proof system~$Q$. Informally, the reflection principle expresses the
soundness of~$Q$: every formula that has a $Q$-refutation is unsatisfiable.

The main result of this subsection is Theorem~\ref{thm:reflection_theorem}. It states
that if $V^0$ proves the reflection principle restricted to refutations of CNF formulas for~$Q$, then
there exists a constant~$d$ such that the bounded-depth Frege system~$F_d$
p-simulates~$Q$ on this class of proofs. The proof of this theorem relies on the
Paris--Wilkie theorem for $V^0$ (Theorem~\ref{thm:paris_wilkie_v0}). In the subsequent sections, we will apply this result to the
proof system~$P_{CA}$ to be introduced in Section~\ref{sec:lower_bound}.

Let $\varphi$ be a CNF formula with propositional variables $p_1, p_2, \dots, p_x$ and clauses $C_1, C_2, \dots, C_y$, $w$ be its $Q$-refutation of size $\leq z$ and we assume $x,y \le z$.

The provability relation 
$Q$ is polynomial-time decidable by Definition~\ref{def:proof_system} and hence, in particular, is in 
the class $\mathrm{NP}$. By Fagin’s theorem, there exist relation symbols $G_1,\ldots,G_t$ (of some arities) on $[z+1]$ 
and a $\Delta_0(w,\varphi,G_1,\ldots,G_t)$-formula
\[
\mathrm{Refut}^0_Q(x,y,z,w,\varphi,G_1,\ldots,G_t)
\]
such that, for
\[
\mathrm{Refut}_Q(x,y,z,w,\varphi) := \exists G_1 \ldots G_t \, 
  \mathrm{Refut}^0_Q(x,y,z,w,\varphi,G_1,\ldots,G_t)
\]
and, for all $x := n$, $y := m$, $z := s$, every CNF formula $\varphi$ with $\le n$ propositional variables 
and size $\le m$ and every string $w$ it holds that
\[
w \text{ is a $Q$-refutation of $\varphi$ of size } \le s 
\quad\text{if and only if}\quad 
\mathrm{Refut}_Q(n,m,s,w,\varphi).
\]

Moreover, in Subsection~\ref{subsub:coding_cnfs} we already introduced predicates:

\begin{itemize}
  \item a $\Delta_0(\varphi)$-formula $\mathrm{CNF}(x,y, \varphi)$ formalizing that $\varphi$ encodes a
  CNF formula with $\le x$ propositional variables and of size $\le y$;

  \item a $\Delta_0(\tuple a,\varphi)$-formula $\mathrm{Sat}(x,y, \tuple a, \varphi)$ formalizing that $\tuple a$ is a truth
  assignment to the $\le x$ propositional variables that satisfies the formula $\varphi
  $;
\end{itemize}

\begin{theorem}[From Reflection to Simulation {\cite[Section~8.4]{krajicek2019}}]\label{thm:reflection_theorem}
    Let $Q$ be an arbitrary proof system and let $\mathrm{Refut}_Q$ be an arbitrary
$NP$-definition of its refutation predicate for CNF formulas of the form as above.

If $V^0$ proves 
\[
\mathrm{CNF}(x,y, \varphi) \wedge \mathrm{Refut}_Q^0(x,y,z,w,\varphi,G_1,\ldots,G_t) \;\rightarrow\; \neg\mathrm{Sat}(x,y,\tuple a,\varphi),
\]
then there exists $d$ such that $F_d$ p-simulates $Q$ with respect to refutations of CNF formulas.
\end{theorem}

This theorem is the key to our proof of the lower bound in Theorem~\ref{thm:main_lower_bound} and thus,
for the benefit of the reader, we shall explain why it holds.

Assume the hypothesis of the theorem. By Theorem~\ref{thm:paris_wilkie_v0},
the propositional translations of
\[
\CNF \wedge \mathrm{Refut}_Q^0 \;\rightarrow\; \neg \mathrm{Sat}
\]
have polynomial-size proofs in $F_d$, for some fixed constant $d$.

Now fix a concrete CNF formula $\varphi$.
Suppose $\pi$ is a $Q$-refutation of $\varphi$,
with witnesses $G_1,G_2,\dots,G_t$ showing that $\pi$ is correct.
From this data we obtain an $F_d$-refutation of the propositional translation
with parameters determined by $\varphi$, $\pi$, and the relations $G_i$.

In this refutation we substitute concrete bits:
for propositional variables describing $\varphi$, $\pi$, and $G_i$,
we plug in the bits coding these objects.
After this substitution, the propositional translation of
\[
\CNF \wedge \mathrm{Refut}_Q^0
\]
becomes a true propositional sentence (it has no variables),
since $\pi$ is indeed a correct $Q$-refutation of $\varphi$.

Any true propositional sentence has a short proof,
because we can simply evaluate it.
Using modus ponens, we then obtain a proof of the formula
\[
\langle\neg \mathrm{Sat}(x,y,\tuple a,\varphi)\rangle_{n,m}.
\]
This formula contains only the propositional variables corresponding to the unknown bits
of the assignment $\tuple a$
(the propositional variables describing $\varphi$ were already replaced by constants).

From this we derive
\[
\neg \varphi(\tuple p).
\]

The whole construction can be carried out by a polynomial-time algorithm.
Therefore we obtain a p-simulation.

The full technical details are in~\cite[Section~8.4]{krajicek2019}.

\section{Computation table for a formula}

In Section~\ref{sec:from_cnf_to_ca} we aim to prove the following theorem:

\begin{theorem}
    There exists a polynomial algorithm that computes from a CNF formula $\varphi$ a cellular automaton $A_\varphi$ and $n, m \geq 1$ such that  $\varphi$ is satisfiable if and only if $A_\varphi$ is not injective on configurations inside $[n+1] \times [m+2]$. 
\end{theorem}

In this section, we introduce computation tables associated with a given CNF formula $\varphi$ and an assignment $\tuple a$.
In the next section, these computation tables will be used to formally define the set of states
and the bounded configurations of a cellular automaton.
The automaton $A_\varphi$ will then verify whether a given configuration encodes a
correct computation of the value of $\varphi$ on a given tuple $\tuple a$.

The crucial property of this construction is the following:
the automaton $A_\varphi$ is injective on finite configurations of particular size if and only if the encoded computation
contains an error or the formula $\varphi$ is not satisfiable.

For the remainder of the paper, we fix $n,m \in \mathbb{N}$, a CNF formula
\[
\varphi = C_1 \wedge \cdots \wedge C_n
\]
over $x_1,\dots,x_m$, and an assignment
\[
\tuple a = (a_1,\dots,a_m) \in \{0,1\}^m.
\]

For convenience, we denote by $C_{i,j}$ the restriction of the clause $C_i$ to the variable $x_j$:

\[
C_{i,j} \;=\;
\begin{cases}
    x_j & \text{if } x_j \in C_i, \\[6pt]
    \neg x_j & \text{if } \neg x_j \in C_i, \\[6pt]
    \bot & \text{otherwise.}
\end{cases}
\]

With a pair \((\varphi, \tuple{a})\) we associate a unique table of size
$(n+1) \times (m+2)$, which we call the \emph{computation table}.
Its cells are indexed by pairs $(i,j)$, where
$i \in [n+1]$ and $j \in [m+2]$.


\bigskip

\subsection{Types of cells}\label{types_of_cells}

The computation table consists of six types of cells.
Every cell stores its coordinates $(i,j)$ and a label from $\{0,1\}$, chosen
arbitrarily. Depending on its position, a cell may also carry additional
information, as described below.

\begin{itemize}    
    \item The \textbf{top-left cell} at $(0,0)$, the \textbf{$0$th column} cells at $(i,0)$ \linebreak for $i \in \{1,2,\dots, n\}$, and the \textbf{top-right cell} at $(0, m+1)$ carry no information additional to coordinates and label.
    
    \item \textbf{$0$th row.} Cells at positions $(0,j)$ for $j \in \{1,2,\dots,m\}$ store the value
    $a_j \in \{0,1\}$ of the variable $x_j$ under the assignment $\tuple a$.
    
    \item  \textbf{Last column.} Cells at positions $(i, m+1)$
    for $i \in \{1,2, \dots ,n\}$, store the value of the partial conjunction
$C_1 \wedge C_2 \wedge \dots \wedge C_i$ evaluated under~$\tuple a$, i.e. elements of $\{0,1\}$. 

Notice that the cell in position $(n, m+1)$ contains the value of $\varphi(\tuple a)$. 
    \item \textbf{Main body.} For each $i \in \{1,2,\dots,n\}$ and $j \in  \{1,2,\dots,m\}$ the cell in position $(i,j)$ contains:
\begin{itemize}
         \item the value $a_j \in \{0,1\}$ of the variable $x_j$ under $\tuple a$;
  \item a \emph{variable flag}
  \[
  \mathrm{flag}(i,j) \;=\;
  \begin{cases}
    1 & \text{if } x_j \in C_i, \\[6pt]
    -1 & \text{if } \neg x_j \in C_i, \\[6pt]
    0 & \text{otherwise; }
  \end{cases}
  \]
  \item the value of the partial disjunction
  \[
  \bigl(\bigvee_{u \leq j} C_{i,u}\bigr) (\tuple a),
  \]
    that is, the truth value of the subformula of \(C_i\) restricted to the literals over variables \(x_1, \dots, x_j\). 


\end{itemize}

\end{itemize}

\begin{table}[h]
\centering
\resizebox{\columnwidth}{!}{%
\begin{tabular}{|c|c|c|c|c|}
 \hline
$((0,0), *)$
& $((0,1), *, a_1=0)$
& $((0,2), *, a_2=1)$
& $((0,3), *, a_3=1)$
& $((0,4), *)$ \\
 \hline
$((1,0), *)$
& $((1,1), *, a_1=0,\; \mathrm{flag}_{1,1}=1,\; \vee_{1,1}=0)$
& $((1,2), *, a_2=1,\; \mathrm{flag}_{1,2}=0,\; \vee_{1,2}=0)$
& $((1,3), *, a_3=1,\; \mathrm{flag}_{1,3}=0,\; \vee_{1,3}=0)$
& $((1,4), *, \wedge_1=0)$ \\
 \hline
$((2,0), *)$
& $((2,1), *, a_1=0,\; \mathrm{flag}_{2,1}=0,\; \vee_{2,1}=0)$
& $((2,2), *, a_2=1,\; \mathrm{flag}_{2,2}=1,\; \vee_{2,2}=1)$
& $((2,3), *, a_3=1,\; \mathrm{flag}_{2,3}=-1,\; \vee_{2,3}=1)$
& $((2,4), *, \wedge_2=1)$ \\
 \hline
$((3,0), *)$
& $((3,1), *, a_1=0,\; \mathrm{flag}_{3,1}=-1,\; \vee_{3,1}=1)$
& $((3,2), *, a_2=1,\; \mathrm{flag}_{3,2}=0,\; \vee_{3,2}=1)$
& $((3,3), *, a_3=1,\; \mathrm{flag}_{3,3}=-1,\; \vee_{3,3}=1)$
& $((3,4), *, \wedge_3=1)$ \\
 \hline
$((4,0), *)$
& $((4,1), *, a_1=0,\; \mathrm{flag}_{4,1}=1,\; \vee_{4,1}=0)$
& $((4,2), *, a_2=1,\; \mathrm{flag}_{4,2}=1,\; \vee_{4,2}=1)$
& $((4,3), *, a_3=1,\; \mathrm{flag}_{4,3}=0,\; \vee_{4,3}=1)$
& $((4,4), *, \wedge_4=1)$ \\
 \hline
\end{tabular}}
\caption{Computation table for $\varphi = C_1 \wedge C_2 \wedge C_3 \wedge C_4$, where $C_1 = x_1$, $C_2 = x_2 \vee \neg x_3$, $C_3 = \neg x_1 \vee \neg x_3$, and $C_4 = x_1 \vee x_2$. The assignment is $\tuple a = (0,1,1)$. We use the abbreviations $\mathrm{flag}_{i,j}$ for $\mathrm{flag}(i,j)$, $\vee_{i,j}$ for $\bigl(\bigvee_{u \le j} C_{i,u}\bigr)(\tuple a)$, and $\wedge_i$ for $\bigl(\bigwedge_{v \le i} C_v\bigr)(\tuple a)$. The symbol $*$ denotes an arbitrary label in $\{0,1\}$.}
\label{tab:sample_formula_table}
\end{table}

\section{From CNF to cellular automaton}\label{sec:from_cnf_to_ca}

We now turn the computation tables introduced in the previous section into objects
that can be checked locally by a cellular automaton.

We define a cellular automaton $A_{\varphi}$ that will operate inside~$[n+1]\times[m+2]$ whose purpose is to verify whether a given configuration represents a correct computation of the truth value of $\varphi$ for some assignment $\tuple a$. These configurations are intended to encode computation tables
for $\varphi$ as defined in the previous section. However, the cells will contain not only the information used for the verification of the correctness but also an additional bit called \emph{label}.

The idea is that $A_{\varphi}$ performs a local verification of the correctness. It inspects the neighbourhood of each cell and checks  whether the local configuration is consistent with a correct computation. If yes, the automaton updates the label component of the cell's state according to the labels of its neighbors. The consistency requirements are formalized by a collection of \emph{local correctness rules}, which will be specified later in this section.  On each cell $A_{\varphi}$  will possibly modify only the label component of the cell state, leaving all other symbols unchanged.

In order to define the cellular automaton $A_\varphi$, we first formalize the notions
introduced in the previous section and fix some of its parameters.
The remaining components will be specified later.

In particular,  we fix:
\begin{itemize}
    \item the dimension, which is $d=2$;
    \item the neighbourhood, which is the \emph{von Neumann neighbourhood}
    \[
        N = \{(0,0), (0,1), (0,-1), (1,0), (-1,0)\}.
    \]
\end{itemize}

For a cell with coordinates $(i,j)$ we call $(i,j+1)$ the right neighbour, $(i,j-1)$ the left neighbour, $(i+1,j)$ the below neighbour, and $(i-1,j)$ the above neighbour.

The local transition function $f$ and the set of states $S_{n,m}$ will be defined in the following subsections.

\subsection{Set of states}\label{states}

To define the \emph{set of states}, we generalize the six types of cells described in
Subsection~\ref{types_of_cells}. Each cell state is a tuple
\[
s = \bigl((i,j),\, \mathrm{flag},\, a,\, \mathsf{pd},\, \mathsf{pc},\, \mathrm{label}\bigr),
\]
whose components encode the information stored in a state of a single cell in a configuration. We use the symbol $\square$ to indicate that a component is not present.
In particular, this allows us to define the \emph{quiescent state}
\[
\tuple q = ((\square, \square), \square, \square, \square, \square, \square).
\]
We describe the meaning of the individual components of a state $s$ as follows:
\begin{itemize}
    \item \textbf{Coordinates.}
    The component $\coord(s)=(i,j)\in ([n+1]\cup \square)\times([m+2]\cup \square)$ gives the position of the cell
    in the computation table.

    \item \textbf{Variable flag.}
    The component $\mathrm{flag}(s)\in\{-1,0,1,\square\}$ is intended to encode whether
    the literal $x_j$, $\neg x_j$, or neither occurs in clause $C_i$.

    \item \textbf{Assignment value.}
    The component $a(s)\in\{0,1,\square\}$ is intended to store the value $a_j$ of the
    variable $x_j$ under the assignment $\tuple a$.

    \item \textbf{Partial disjunction.}
    The component $\mathsf{pd}(s)\in\{0,1,\square\}$ is intended to represent the truth
    value of the partial disjunction
    \[
    \bigl(\bigvee_{u \le j} C_{i,u}\bigr)(\tuple a).
    \]

    \item \textbf{Partial conjunction.}
    The component $\mathsf{pc}(s)\in\{0,1,\square\}$ is intended to represent the truth
    value of the partial conjunction
    \[
    \bigl(\bigwedge_{v \le i} C_v\bigr)(\tuple a).
    \]

    \item \textbf{Label.}
    The component $\mathrm{label}(s)\in\{0,1,\square\}$ is an auxiliary label used by the
    cellular automaton.
\end{itemize}

We emphasize that these components store the \emph{written values} in a configuration.
They may be incorrect. The role of the cellular automaton $A_\varphi$ is to locally verify that the state of
each cell is consistent with the states of its neighbours in the von Neumann
neighbourhood.

Note that we refer to the components of a state $s$ using the notation \linebreak
$\textit{component}(s)$. For example, $\mathrm{flag}(s)$ denotes the value stored in
the component representing the variable flag. Formally, each component is a function on the set of states $S_{n,m}$ (defined below)
with values in the domain of that component. That is, 
\[
\flag: S_{n,m} \to \{-1, 0, 1, \square\},
\]

and similarly for the other components.

We now define seven subsets of the set of states $S_{n,m}$: six corresponding to the internal cell types described in Subsection~\ref{types_of_cells}, and one representing the quiescent state surrounding the finite part of the configuration. These subsets are given as follows:

\begin{enumerate}
    \item \textbf{Top-left cell:} 
    \[
    S_1 = \{ ((0,0),  \square, \square, \square, \square, \mathrm{label}) \mid \mathrm{label} \in \{0, 1\} \}
    \]
    
    \item \textbf{$0$th column:} 
    \[
    S_2 = \{ ((i,0),\square, \square, \square, \square, \mathrm{label}) \mid i \in \{1, 2, \dots, n \}, \mathrm{label} \in \{0, 1\}\}
    \]
    
    \item \textbf{$0$th row:} 
    \[
    S_3 = \{ ((0, j), \square, a, \square, \square, \mathrm{label}) \mid j \in \{1, 2, \dots, m \}, a, \mathrm{label} \in \{0, 1\} \}
    \]

    \item \textbf{Top-right cell:} 
    \[
   S_4 = \{  ((0, m+1), \square, \square, \square, \square, \mathrm{label}) \mid \mathrm{label} \in \{0, 1\} \}
    \]

    \item  \textbf{Last column:} 
    \[
   S_5 = \{  ((i, m+1), \square, \square, \mathsf{pc}, \mathrm{label}) \mid i \in \{1, 2, \dots, n\},\mathsf{pc} \in \{0, 1\}, \mathrm{label} \in \{0, 1\} \}
    \]

    The cell with state $s$ such that $\coord(s)=(n,m+1)$ is called an \emph{output cell}. Notice that this notion refers only to the written coordinates. Since coordinates are part of the written data and may be incorrect,
a configuration may contain several output cells, or none at all.
    
    \item \textbf{Main body:} 
   \[
\begin{aligned}
S_6 = {} & \{ ((i, j), \mathrm{flag}, a, \mathsf{pd}, \square, \mathrm{label}) 
   \mid i \in \{1, \dots, n\},\, j \in \{1, \dots, m\}, \\
   & \quad \mathrm{flag}, a, \mathsf{pd}, \mathrm{label} \in \{0, 1\} \} 
\end{aligned}
\]


    \item \textbf{Quiescent states around the finite part of the configuration:}
    \[
    S_7 = \{\tuple q \} \text{, where } \tuple q = ( (\square, \square), \square, \square, \square, \square, \square)
    \]

\end{enumerate}

The set of states is:
\begin{align}\label{eq:set_of_states_bound}
 S_{n,m} = & S_1 \cup S_2 \cup S_3 \cup S_4 \cup S_5 \cup S_6 \cup S_7  \\ & \subseteq \nonumber([n+1] \cup \{\square\}) \times([m+2]\cup \{\square\})\times\{-1, 0, 1, \square\} \times \{0,1, \square\}^4
\end{align}

Again we distinguish between the values that a cell is \emph{supposed to represent}
and the values that are \emph{actually written} in the computation table.
The written values are part of the cell's state and may be incorrect. The names of subsets $S_1, S_2, S_3, S_4, S_5, S_6$ are derived from the written value.
For example, a cell with a state $s$ such that $\coord(s) = (0,0)$ can be situated in the bottom part of the configuration, but we still call it the top-left cell.

We denote by $\Config_{n,m}$ the set of all finite configurations
inside $[n+1]\times [m+2]$ such that every cell in this rectangle
has a state from $S_{n,m}$ and no cell in the finite part is in the
quiescent state.

From now on, whenever we speak about a (finite) configuration, we always
mean an element of $\Config_{n,m}$.

Since $A_\varphi$ can only inspect the neighbourhood of each cell, it verifies
correctness using a collection of \emph{local correctness rules}.
These rules check, for example, that the values $a_j$ are constant along each
column, that the coordinates $(i,j)$ change in the expected way when moving
vertically or horizontally in the configuration, and that the values of the partial
disjunctions in the main body and the partial conjunctions in the last column are
computed correctly from neighbouring cells. We also impose several "technical conditions" on the cells as not all cells are intended to contain all state components. 
For example, the cell with a state $s$ such that $\coord(s) = (0,0)$ is required to contain only its coordinates, while all
other state components in this cell are set to $\square$.

\begin{definition}[Local correctness]\label{def:locally_correct}
Assume that all states mentioned below belong to $S_{n,m}$, and let 
$C \in \Config_{n,m}$ be a configuration. 

Let $c$ be a non-quiescent cell with state $s$ such that $\coord(s) = (i,j)$, where $i \in [n+1]$ and $j \in [m+2]$.  

We say that $c$ is \emph{locally correct} if the
following conditions hold. These conditions are
called the \emph{local correctness rules}. 

\begin{enumerate}[label=(\Alph*)]
\item \textbf{Index consistency.}

Let $\coord(s)=(i,j)$ and let $s_1,s_2,s_3,s_4$ be the states of the
right, left, below, and above neighbours of $c$, respectively.

If $i \in \{1,\dots,n-1\}$ and $j \in \{1,\dots,m\}$ (that is, $c$ is
not on the boundary), then
\[
\coord(s)=(i,j)
\]
if and only if
\begin{align*}
\coord(s_1) &= (i,j+1) \;\wedge\;
\coord(s_2) = (i,j-1) \;\wedge \\
\coord(s_3) &= (i+1,j) \;\wedge\;
\coord(s_4) = (i-1,j).
\end{align*}

If $(i,j)$ lies on the boundary (that is, $i=0$ or $i=n$ or
$j=0$ or $j=m+1$), then the same equivalence holds for the neighbours
inside the table, while every neighbour outside the finite part is in
the quiescent state $\tuple q$.

For example,
\[
\coord(s)=(0,0)
\]
if and only if
\begin{align*}
\coord(s_1) &= (0,1) \;\wedge\;
s_2 = \tuple q \;\wedge \\
\coord(s_3) &= (1,0) \;\wedge\;
s_4 = \tuple q.
\end{align*}

\item \textbf{Formula consistency.}
In the main body, the $\flag$ component must correctly encode the
occurrence of variables in the formula $\varphi$.

More precisely, if 
$i \in \{1,2, \dots,n\}$ and $j \in \{1,2,\dots,m\}$, then:
\[
\begin{aligned}
\flag(s) = 1 
&\;\leftrightarrow\; x_j \in C_i, \\
\flag(s) = -1 
&\;\leftrightarrow\; \neg x_j \in C_i, \\
\flag(s) = 0 
&\;\leftrightarrow\; \bigl(x_j \notin C_i \,\wedge\, \neg x_j \notin C_i\bigr).
\end{aligned}
\]

\item \textbf{Vertical consistency.}
\begin{enumerate}[label=(C\arabic*)]
 \item If the column $j \in \{1,2,\dots,m\}$, then the $a$-value is constant along
the column. That is, for every neighbour $c'$ above or below $c$
(with state $s'$),
\[
a(s') = a(s).
\]
\item If $(i,j)=(r,m+1)$ with $r\ge 2$, let $c'$ be the left neighbour
with state $s'$ and $c''$ the upper neighbour with state $s''$.
Then the partial conjunction satisfies
\[
\pc(s)
=
\pc(s'') \wedge \pd(s').
\]
Equivalently:
\begin{itemize}
    \item if $\pc(s'')=0$, then $\pc(s)=0$;
    \item if $\pc(s'')=1$, then $\pc(s)=1$ if and only if $\pd(s')=1$.
\end{itemize}

  \item For $i=1$, the base case is checked directly. This means that if $(i,j)=(1,m+1)$, let $c'$ be the left neighbour with state
$s'$. Then
\[
\pc(s)=1 \;\leftrightarrow\; \pd(s')=1.
\]
\end{enumerate}
\item \textbf{Horizontal consistency.}
\begin{enumerate}[label=(D\arabic*)]
  \item If $1 \leq i \leq n$ and $2 \leq j \leq m$. Let $c'$ be the
left neighbour with state $s'$. Then
\[
\pd(s)
=
\bigl(\pd(s')
\;\vee\;
(\flag(s)=1 \wedge a(s)=1)
\;\vee\;
(\flag(s)=-1 \wedge a(s)=0)\bigr).
\]

Equivalently:
\begin{itemize}
    \item if $\pd(s')=1$, then $\pd(s)=1$;
    \item if $\pd(s')=0$, then $\pd(s)=1$ if and only if 
    $\flag(s)=1 \wedge a(s)=1$ or 
    $\flag(s)=-1 \wedge a(s)=0$.
\end{itemize}

  \item If $(i,j)=(r,1)$ with $r\in\{1,\dots,n\}$, then
\[
\pd(s)=1
\;\leftrightarrow\;
\bigl(\flag(s)=1 \wedge a(s)=1\bigr)
\;\vee\;
\bigl(\flag(s)=-1 \wedge a(s)=0\bigr).
\]
\end{enumerate}

\item \textbf{Technical conditions.}
Depending on the position $(i,j)$, the unused components of $s$
must be $\square$:
\begin{enumerate}[label=(E\arabic*)]
\item If $j=0$, or $(i,j)=(0,0)$, or $(i,j)=(0,m+1)$,
all components except coordinates and label are $\square$.
\item If $i=0$ and $1\le j\le m$, all components except
coordinates, $a$, and label are $\square$.
\item If $j=m+1$ and $i\ge 1$, all components except
coordinates, $\pc$, and label are $\square$.
\item If $1\le i\le n$ and $1\le j\le m$, then $\pc(s)=\square$.
\end{enumerate}

\end{enumerate}

Each cell class (top-left, $0$th row, etc.) enforces the relevant subset of
(A)--(E).
\end{definition}

In the following Lemma~\ref{lm:new_lemma} we show that if all cells in a configuration $C \in \Config_{n,m}$
satisfy the local correctness rules, then the configuration encodes a correct
computation of the value of $\varphi$ on $\tuple a$. 

In particular, by correct computation we mean that for each 
$(i,j) \in [n+1] \times [m+2]$ and each cell $c=(i,j,s)\in C$ it holds that
\begin{align}\label{eq: definition_table}
\coord(s) &= (i,j), \nonumber\\[6pt] 
\flag(s) &=
\begin{cases} \nonumber
k
& \text{if } (i,j)\in \{1,\dots,n\}\times\{1,\dots,m\} \\
& \quad \text{and } (i,j,k)\in \varphi,\; k\in\{-1,0,1\}, \\[6pt]
\square
& \text{otherwise,}
\end{cases} \\[10pt]
a(s) &=
\begin{cases} \nonumber
a_j
& \text{if } (i,j)\in [n+1]\times\{1,\dots,m\}, \\[6pt]
\square
& \text{otherwise,}
\end{cases} \\[10pt]
\pd(s) &=
\begin{cases} \nonumber
\bigl(\bigvee_{u\le j} C_{i,u}\bigr)(\tuple a)
& \text{if } (i,j)\in \{1,\dots,n\}\times\{1,\dots,m\}, \\[6pt]
\square
& \text{otherwise,}
\end{cases} \\[10pt]
\pc(s) &=
\begin{cases}
\bigl(\bigwedge_{v\le i} C_v\bigr)(\tuple a)
& \text{if } i\in \{1,\dots,n\},\; j=m+1, \\[6pt]
\square
& \text{otherwise.}
\end{cases}
\end{align}

Such configuration, however, is not unique: the labels of the cells are still arbitrary,
and therefore there are exponentially many configurations that contain the same correct
computation of $\varphi$ on $\tuple a$. We fix one canonical representative by requiring that every cell has label $0$ 
and denote this configuration by $\Table_\varphi(\tuple a)$.

For configurations $C_1, C_2 \in \Config_{n,m}$, we write $C_1 \sim C_2$ if they agree on all components of their states except, possibly, the label component. In this case, we say
that $C_1$ and $C_2$ are \emph{similar}. In particular, a configuration similar to $\Table_\varphi(\tuple a)$ encodes the correct computation, too, but may have different labels than just zeros. Clearly $\sim$ is an equivalence relation on the set $\Config_{n,m}$.

\begin{lemma}\label{lm:table_loc_correct}
    Let $C\in \Config_{n,m}$ such that $C \sim \Table_\varphi(\tuple a)$. Then $C$ is locally correct. In particular, $\Table_\varphi(\tuple a)$ is locally correct.
\end{lemma}
\begin{proof}
Fix a configuration $C$ such that $C \sim \Table_\varphi(\tuple a)$. We show that each cell of $C$ satisfies the local correctness rules.

 Clearly, the statement holds for the $\coord$ and $a$ components from the definition. Moreover, for $\flag$ the condition is exactly the same as in local correctness rules.

 Fix $(i,j) \in \{1,2,\dots,n\}\times\{1,2,\dots,m\}$. Then $c = (i,j,s) \in C$ for some $s \in S_{n,m}$ and $\pd(s) =\bigl(\bigvee_{u\le j} C_{i,u}\bigr)(\tuple a)$. Assume first that $j\geq 2$. Then
 \begin{align*}
      \pd(s) =&\bigl(\bigvee_{u\le j} C_{i,u}\bigr)(\tuple a)  = \bigl(\bigvee_{u\le j-1} C_{i,u}\bigr)(\tuple a) \vee C_{i,j}
 (\tuple a) \\&= \pd(s') \vee (\flag(s)=1 \wedge a(s)=1)
\;\vee\;
(\flag(s)=-1 \wedge a(s)=0),
 \end{align*}
where $s'$ is the state of cell $c' = (i, j-1, s')\in C$, which is exactly the left neighbour of $c$.

If $j=1$, then
 \begin{align*}
      \pd(s) =&C_{i,1}(\tuple a)  =  (\flag(s)=1 \wedge a(s)=1)
\;\vee\;
(\flag(s)=-1 \wedge a(s)=0).
 \end{align*}

 Assume now that $i\ge 2$ and $j=m+1$. Then
 \begin{align*}
     \pc(s) = \bigl(\bigwedge_{v\le i} C_v\bigr)(\tuple a)  = \bigl(\bigwedge_{v\le i-1} C_v\bigr)(\tuple a) \wedge C_i(\tuple a) = \pc(s') \wedge \pd(s''),
 \end{align*}
    where $s'$ is the state of cell $c' = (i-1, m+1, s') \in C$ (upper neighbour of $c$) and $s''$ is the state of cell $c'' = (i, m, s'') \in C$ (left neighbour of $c$).

    The cases where $\pc(s)=\square$ or $\pd(s) = \square$ coincide with the same cases in local correctness rules. 

    All local correctness rules are satisfied in
$C$.
\end{proof}

Moreover, the converse implication also holds: if a configuration is locally correct and its $0$th row encodes the assignment $\tuple a$, then the whole configuration is uniquely determined except for the label components. In other words, local correctness forces the configuration to be similar to $\Table_\varphi(\tuple a)$.

\begin{lemma}\label{lm:new_lemma}
Let $\tuple a=(a_1,a_2,\dots,a_m)\in\{0,1\}^m$ and let
$C\in\Config_{n,m}$.
If $C$ is locally correct and, for every $j\in\{1,2,\dots,m\}$, every cell
with state $s$ satisfying $\coord(s)=(0,j)$ also satisfies $a(s)=a_j$, then
$C$ is similar to $\Table_\varphi(\tuple a)$.
\end{lemma}

\begin{proof}
We show that every component of the state of each cell apart from label coincides with the corresponding entry
of $\Table_\varphi(\tuple a)$.

\paragraph{Coordinates.}
We show that for every cell $c=(i,j,s)$ with 
$(i,j)\in [n+1]\times[m+2]$ it holds
\[
\coord(s)=(i,j).
\]

We first prove correctness of the row index by induction on $i$,
using local correctness rule~(A).

\medskip
\noindent
\emph{Base case ($i=0$).}
Let $c=(0,j,s)$ for some $j\in[m+2]$, and let $c'$ be the cell above
$c$. Since $C\in\Config_{n,m}$, the cell $c'$ is outside the finite
part of the table and therefore is in the quiescent state $\tuple q$.
By index consistency~(A), a cell whose above neighbour is in state
$\tuple q$ must have first coordinate $0$. Hence
\[
\coord(s)=(0,j)
\]
for some $j$, and therefore the row index is correct.

\medskip
\noindent
\emph{Induction step.}
Assume that for some $i\ge 0$ every cell in rows $0, 1, \dots, i$ has correct
first coordinate in its state. 

Let $c=(i+1,j,s)$ and let $c'=(i,j,s')$ be the
cell above $c$.  By inductive assumption $\coord(s') = (i,j)$. By rule~(A), as $c$ is below neighbour of $c'$, $\coord(s) =(i+1, j)$. 
Thus every cell in row $i+1$ has the correct first coordinate.
This completes the induction on $i$.

An analogous induction on the column index $j$ shows that the second
coordinate is also correct. 

In the remainder of the proof, we rely on each cell’s state containing its correct coordinates.

\paragraph{Assignment values.}
By assumption, the $0$th row contains the correct assignment values $a_j$.
By vertical consistency~(C1), these values are propagated unchanged to all cells
below the $0$th row. By an easy induction on row index, every cell with coordinates $(i,j)$ for $i \in \{1, 2, \dots, n\}$ and $j \in \{1, 2, \dots, m\}$ contains the
correct value $a_j$.

\paragraph{Partial disjunctions.}
We show that for every $i\in\{1,\dots,n\}$ and every $j\in\{1,\dots,m\}$,
every cell with state $s$ such that $\coord(s)=(i,j)$ satisfies
\[
\mathsf{pd}(s)=\bigl(\bigvee_{u\le j} C_{i,u}\bigr)(\tuple a).
\]
The proof is by induction on~$j$.

\emph{Base case $j=1$.}
By rule~(D2), $\mathsf{pd}(s)=1$ if and only if either
$\flag(s)=1$ and $a_1=1$, or $\flag(s)=-1$ and $a_1=0$.
In both cases, the literal in column~$1$ satisfies clause~$C_i$,
so $C_{i,1}(\tuple a)=1$.
If none of these cases holds, then either $x_1$ does not satisfy $C_i$
or does not occur in it, and hence $C_{i,1}(\tuple a)=0$.
Thus $\mathsf{pd}(s)=C_{i,1}(\tuple a)$.

\emph{Induction step.}
Assume the claim holds for column $k-1$, and let $s$ be a state such that
$\coord(s) = (i,k)$.  
Let $s'$ be the state of its left neighbour.

Assume first $\mathsf{pd}(s')=1$. (D1) implies that $\mathsf{pd}(s)=1$. Moreover, by inductive assumption $\bigl(\bigvee_{u \leq k-1} C_{i,u}\bigr)(\tuple a) = \pd(s') = 1$, which means  
\[
\bigl(\bigvee_{u \leq k} C_{i,u}\bigr)(\tuple a) = \bigl(\bigvee_{u \leq k-1} C_{i,u}\bigr)(\tuple a) \vee C_{i, k} (\tuple a)= 1 \vee C_{i, k} (\tuple a) =1.
\]
 Therefore $\mathsf{pd}(s) = \bigl(\bigvee_{u \leq k} C_{i,u}\bigr)(\tuple a)$ in this case.

Assume $\pd(s')=0$. Here (D1) yields that $\mathsf{pd}(s)=1$ if and only if $\flag(s)=1\wedge a(s)=1$ or $\flag(s)=-1\wedge a(s)=0$. Similarly as in the base step, if $\mathsf{pd}(s)=1$, then in both cases $a_{k}$ satisfies $C_i$ and 
 \[
 \bigl(\bigvee_{u \leq k} C_{i,u}\bigr)(\tuple a) = \bigl(\bigvee_{u \leq k-1} C_{i,u}\bigr)(\tuple a) \vee C_{i,k}(\tuple a) = \bigl(\bigvee_{u \leq k-1} C_{i,u}\bigr)(\tuple a) \vee 1=1.
 \]
   If $\mathsf{pd}(s)=0$, then either $a_{k}$ does not satisfy or does not belong to $C_i$. In both of these cases, as $\bigl(\bigvee_{u \leq k-1} C_{i,u}\bigr)(\tuple a) =0$ by inductive assumption, then 
    \[
 \bigl(\bigvee_{u \leq k} C_{i,u}\bigr)(\tuple a) = \bigl(\bigvee_{u \leq k-1} C_{i,u}\bigr)(\tuple a) \vee C_{i,k}(\tuple a) = 0 \vee 0=0,
 \]
   
 and we are done.

\paragraph{Partial conjunctions.}
Finally, we show that for every $i\in\{1,\dots,n\}$ and every cell with $s$ and
$\coord(s)=(i,m+1)$,
\[
\mathsf{pc}(s)=\bigl(\bigwedge_{v\le i} C_v\bigr)(\tuple a).
\]
The proof is by induction on~$i$.

\emph{Base case $i=1$.}
By (C3) $\mathsf{pc}(s)=1$ if and only if $\mathsf{pd}(s')=1$ for the left neighbour with state $s'$. We have already shown that $\mathsf{pd}(s')=1$ if and only if $\bigl(\bigvee_{u \leq m} C_{1,u}\bigr)(\tuple a) = 1$. It remains to notice that $(\bigvee_{u \leq m} C_{1,u}\bigr)$ is exactly $C_1$ and we are done.

\emph{Induction step.}
Assume the claim holds for row $l-1$, and let $s$ be a state such that
$\coord(s)=(l,m+1)$. Let $s''$ be the state of the above neighbour.

If $\mathsf{pc}(s'')=0$, then according to (C2), $\mathsf{pc}(s)=0$. By inductive assumption $\bigl(\bigwedge_{v \leq l-1}C_v\bigr) (\tuple a) =\mathsf{pc}(s)=0$, therefore
\[
\bigl(\bigwedge_{v \leq l}C_v\bigr) (\tuple a) = \bigl(\bigwedge_{v \leq l-1}C_v\bigr) (\tuple a) \wedge C_{l}(\tuple a) = 0 \wedge C_{l}(\tuple a) =0.
\]

Assume $\mathsf{pc}(s'')=1$. According to (C2) $\mathsf{pc}(s)=1$ if and only if for left neighbour of $c$ with state $s'$ it holds $\mathsf{pd}(s')=1$. Similarly as in the base case, we have shown in the first part of the proof that $\mathsf{pd}(s') = (\bigvee_{u \leq m} C_{l,u}\bigr)(\tuple a) = C_{l}(\tuple a)$. So $\mathsf{pc}(s)=1$ if and only if $C_{l}(\tuple a) =1$. Then 
\[
\bigl(\bigwedge_{v \leq l}C_v\bigr) (\tuple a) = \bigl(\bigwedge_{v \leq l-1}C_v\bigr) (\tuple a) \wedge C_{l}(\tuple a) = 1 \wedge  C_{l}(\tuple a) = C_{l}(\tuple a).
\]
This shows that $\mathsf{pc}(s)=1$ if and only if $\bigl(\bigwedge_{v \leq l}C_v\bigr) (\tuple a) =1$ and we are done.
\end{proof}

To define the local transition function, we must specify how $A_\varphi$ rewrites
the labels.
The idea is simple: if cell's state violates at least one local correctness rule
(that is, it witnesses an error), or if it is the output cell and its state contains the
value $0$, then the label remains unchanged.
In the opposite case its label is rewritten.

\begin{definition}[Blue and red cells]\label{def:blue_red}
Let $C \in \Config_{n,m}$ and let $c$ be a cell.  
We say that $c$ is \emph{blue} if the following holds:
\begin{itemize}
    \item $c$ is locally correct, and
    \item if $c$ is an output cell, then the value in the $\pc$ component of its state is $1$.
\end{itemize}
Otherwise, we say that $c$ is \emph{red}.

We say that a configuration $C \in \Config_{n,m}$ is \emph{blue}
if every cell $c$ in $C$ is blue.
\end{definition}

Note that the color of a cell does not depend on its label; i.e. if we change  some labels of some  cells but leave everything else in the configuration unchanged, then the color of all cells remains the same.

It remains to specify how a label of a blue cell is rewritten.
In this case, the new label is computed from its label and the label of one of the neighbouring cells, where
the chosen neighbour is determined by the coordinates in the state of the cell. We write this formally in the following definitions.

For simplicity, in the following definitions of direction and $\mathrm{suc}(i,j)$, assume that the number of clauses $n$ is odd; therefore, the number of rows in the finite part of the configuration is even. Otherwise, if the number of clauses is even, just modify the formula by adding a "copy" of the last clause.

\begin{definition}[Direction and successor]\label{def:successor}
Let \( C \in \Config_{n,m} \) and let \( c \) be a cell
in state \( s \neq \tuple q \) with \( \coord(s) = (i,j) \), $i \in [n+1], j \in [m+2]$.

We associate with \( c \) a \emph{direction} 
\( d_{ij} \in \{\to,\leftarrow,\uparrow,\downarrow\} \) defined by
\[
d_{ij} =
\begin{cases}
\to 
& \text{if } (i,j) = (0,0) 
  \text{ or } \bigl(i \equiv 0 \!\!\pmod{2} \text{ and } 1 \le j \le m\bigr), \\[6pt]

\leftarrow 
& \text{if } (i,j) = (n,1) 
  \text{ or } \bigl(i \equiv 1 \!\!\pmod{2} \text{ and } 2 \le j \le m+1\bigr), \\[6pt]

\downarrow 
& \text{if } \bigl(i \equiv 1 \!\!\pmod{2} \text{ and } j=1\bigr)
  \text{ or } \bigl(i \equiv 0 \!\!\pmod{2} \text{ and } j=m+1\bigr), \\[6pt]

\uparrow 
& \text{if } i \ge 1 \text{ and } j=0.
\end{cases}
\]

Then $\mathrm{suc}(i,j)$ denotes the pair of coordinates that is
adjacent to $(i,j)$ in the direction $d_{ij}$, that is,

\[
\mathrm{suc}(i,j) =
\begin{cases}
(i, j+1), & \text{if } d_{ij} = \to, \\[6pt]
(i, j-1), & \text{if } d_{ij} = \leftarrow, \\[6pt]
(i+1, j), & \text{if } d_{ij} = \downarrow, \\[6pt]
(i-1, j), & \text{if } d_{ij} = \uparrow.
\end{cases}
\]
\end{definition}

Note that the directions are defined so that, in any configuration similar to $\Table_\varphi(\tuple a)$, they form a snake-like traversal of the entire finite part with coordinates in
$[n+1]\times[m+2]$; see Table~\ref{tab:arrows}.

\begin{table}[h]
\centering
\begin{tabular}{|c|c|c|c|c|}
 \hline
$\to$ & $\to$ & $\to$ & $\to$ &  $\downarrow$ \\
 \hline
$\uparrow$ & $\downarrow$ & $\leftarrow$ & $\leftarrow$ & $\leftarrow$ \\
 \hline
$\uparrow$ & $\to$ & $\to$ & $\to$ & $\downarrow$ \\
 \hline
$\uparrow$ & $\downarrow$ & $\leftarrow$ & $\leftarrow$ & $\leftarrow$ \\
 \hline
$\uparrow$ &  $\to$ & $\to$ & $\to$ & $\downarrow$  \\
 \hline
$\uparrow$ &  $\leftarrow$ & $\leftarrow$ & $\leftarrow$ & $\leftarrow$ \\
 \hline
\end{tabular}
\caption{Directions in a $6\times5$ configuration similar to
$\Table_\varphi(\tuple a)$.}

\label{tab:arrows} 
\end{table}

We will use the fact that each cell has a unique neighbour in the
direction determined by $\mathrm{suc}(i,j)$.
This property is formalized in the following lemma.

\begin{lemma}\label{lm:exactly_one_successor}
    Assume $c \in \Config_{n,m}$ is a blue cell. Let $s$ be the state of $c$ and assume that $\coord(s) = (i,j)$ for some $(i,j) \in [n+1] \times [m+2]$. Then there exists exactly one neighbour $c'$ of $c$ with state $s'$ such that $\coord(s') = \mathrm{suc}(i,j)$.
\end{lemma}
    We call this uniquely determined neighbour $c'$ the 
\emph{successor} of $c$.
\begin{proof}
Since $c$ is blue, it is locally correct and hence satisfies all local
correctness rules. In particular, by rule~(A) (index consistency),
the coordinate component of the state of each non-quiescent neighbour
is uniquely determined by its relative position.

First assume that $(i,j) \in \{1,2,\dots,n-1\} \times \{1,2,\dots,m\}$, that is, $c$ is not on the margin of the configuration.
Then all four von Neumann neighbours of $c$ are non-quiescent.
By rule~(A), the coordinate components of their states are:
\begin{itemize}
    \item right neighbour: $(i,j+1)$,
    \item left neighbour: $(i,j-1)$,
    \item upper neighbour: $(i-1,j)$,
    \item lower neighbour: $(i+1,j)$.
\end{itemize}
Hence exactly one of these neighbours has coordinate component
$\mathrm{suc}(i,j)$, and the statement follows in this case.

Now assume that $c$ lies on the margin of the configuration. Then some of its neighbours may be in the quiescent state.

Again by rule (A) at most one of its neighbours has coordinates $\mathrm{suc}(i,j)$. Moreover, by its definition, $d_{ij}$ never "points" outside the finite part of the configuration and therefore $\mathrm{suc}(i,j)$ always belongs to
$[n+1]\times[m+2]$. 

Thus even on the margin, $c$ has exactly one non-quiescent neighbour
whose coordinate component is $\mathrm{suc}(i,j)$.
\end{proof}

For the next definition, observe that local correctness, as introduced in Definition~\ref{def:locally_correct}, depends only on the state of a cell $c \in \Config_{n,m}$ and on the states of its von Neumann neighbours. Hence it is completely determined by the five states in this neighbourhood.

We therefore define local correctness at the level of states.
Let 
\[
(s_1,s_2,s_3,s_4,s_5) \in S_{n,m}^5,
\]
where $s_1$ is intended to be
the state of a cell, and $s_2, s_3, s_4, s_5$ are intended to be the
states of its right, left, below and above neighbour cells,
respectively.

We say that $s_1$ is \emph{locally correct with respect to}
$(s_2,s_3,s_4,s_5)$ if every cell $c$ with state $s_1$ whose
right, left, below and above neighbours have states
$s_2, s_3, s_4$ and $s_5$, respectively, is locally correct
in the sense of Definition~\ref{def:locally_correct}.

In the same way, the colour of a cell depends only on its state and the states of its neighbours.
Hence we say that $s_1$ is \emph{blue with respect to}
$(s_2,s_3,s_4,s_5)$ if every cell $c$ with state $s_1$ whose
right, left, below and above neighbours have states
$s_2,s_3,s_4,s_5$, respectively, is blue in the sense of
Definition~\ref{def:blue_red}.  Otherwise, we say that $s_1$ is \emph{red with respect to}
$(s_2,s_3,s_4,s_5)$. 

Finally, let $k \in \{2,3,4,5\}$.  
We say that $s_k$ is the \emph{successor of $s_1$ with respect to} $(s_2, s_3, s_4, s_5)$ if every cell $c$ with state $s_1$ whose
right, left, below and above neighbours have states
$s_2,s_3,s_4,s_5$, respectively, has successor with state $s_k$ according to Definition~\ref{def:successor}. This is well defined by Lemma~\ref{lm:exactly_one_successor}, which
states that every blue cell has exactly one successor.

We are now ready to define the local transition function of the automaton.
Its purpose is to modify only the label component of each cell, while
leaving all other components of the state unchanged.

Intuitively, the automaton propagates information along the
successor relation defined above. If a cell is locally correct and
has a successor, then in one step the automaton updates its label
by taking the direct sum of its own
label and the label of its successor. In all other cases, the label
remains unchanged.

Thus, the new label of a cell depends only on its current label and
on the label of its uniquely determined successor.

\begin{definition}[Local transition function $f_\varphi$]\label{def:local_transition_function}
The local transition function
\[
f_\varphi \; \colon \;S_{n,m}^5 \to S_{n,m}
\]
is defined as follows.

Let $(s_1,s_2,s_3,s_4,s_5) \in S_{n,m}^5$.
Then
\[
f_\varphi(s_1,s_2,s_3,s_4,s_5)
=
\begin{cases}
s'
& \parbox[t]{0.6\linewidth}{%
if $s_1$ is blue with respect to $s_2,s_3,s_4,s_5$,} \\[6pt]
s_1 
& \text{if $s_1$ is red with respect to $s_2,s_3,s_4,s_5$,}
\end{cases}
\]
where
\[
s' = (\coord(s_1),\flag(s_1),a(s_1),\pd(s_1),\pc(s_1),\, \mathrm{label}(s_1) \oplus \mathrm{label}(s_k)),
\]
for 
$k \in \{2,3,4,5\}$ such that $s_k$ is the successor of $s_1$ with respect to $(s_2, s_3, s_4, s_5)$.
\end{definition}

 We are now ready to define the cellular automaton $A_\varphi$.

\begin{definition}\label{def:A_varphi}
    The cellular automaton $A_\varphi$ is defined as a quadruplet \linebreak$(2, S_{n,m}, N, f_\varphi)$, where:
    \begin{itemize}
        \item the set of states $S_{n,m}$ is defined in Subsection~\ref{states},
        \item the neighbourhood $N$ is the von Neumann neighbourhood,
        \item the transition function $f_\varphi$ is defined in Definition~\ref{def:local_transition_function}.
    \end{itemize}
\end{definition}

Since $A_\varphi$ acts only on configurations inside $[n+1]\times[m+2]$, it is useful to view $A_\varphi$ as a map $\Config_{n,m} \rightarrow \Config_{n,m}$. 
Since $A_\varphi$ only rewrites the labels, it does not change the colors of the cells, which are determined by the other components of the cell state. 
As a further consequence of the fact that $A_\varphi$ rewrites only labels, for every configuration $C \in \Config_{n,m}$ we have $C \sim A_\varphi(C)$.

Note that $A_\varphi(C)(c)$ denotes the state of the image of cell $c$. This notation will be used in the proof of Lemma~\ref{lm:global_injective}.

\begin{lemma}
    The size of $A_\varphi$ is bounded above by
$O((n m)^5 \log(nm)) \le O(|\varphi|^{11})$, where $|\varphi|$ is the size of the formula.
\end{lemma}

\begin{proof}
According to Equation~\ref{eq:set_of_states_bound}, the size of $S_{n,m}$ is $O(nm)$. As the size of von Neumann neighbourhood is $5$, applying Lemma~\ref{lm:CA_bound} yields that the size of $A_\varphi$ is bounded above by
\[
(O(nm))^5 \log(nm) = O((nm)^5)\log(nm) \leq O(|\varphi|^{11}).
\]
\end{proof}

\begin{lemma}\label{lm:blue_table}

 Fix $C \in \Config_{n,m}$ such that $C$ is blue and let
 \[
 U = \{C' \in \Config_{n,m} \mid C' \sim C\}.
 \]
 Then $A_\varphi : U\to U$ is not injective.
\end{lemma}
\begin{proof}
Consider $C_1\sim C$ and $C_2 \sim C$  that are identical except for the labels:
in $C_1$ every cell has label $0$, and in $C_2$ every cell has label $1$.
Since all cells are blue, each cell updates its label by taking the direct sum of its label and the label of its successor.
In both configurations the resulting labels are $0$ for every cell.
Hence the two configurations have the same image under $A_\varphi$, and the restriction is not injective.
\end{proof}

\begin{definition}
    \emph{Pointed chain}  in a configuration $C \in \Config_{n,m}$ is any sequence of cells $(c_1, \dots, c_s) \in C$, where $s \in \mathbb{N}$, such that:
    \begin{itemize}
        \item $c_i$ is blue for $i \in \{1, \dots, s-1\}$ and red for $i=s$
        \item $c_{i+1}$ is the successor of $c_{i}$ for $i \in \{1, \dots, s-1\}$
    \end{itemize}
\end{definition}

\begin{lemma}\label{lm:global_injective}
Let $C, C' \in \Config_{n,m}$ be configurations such that $C \sim C'$, and let
\[
\gamma = (c_1, c_2, \dots, c_t), \qquad t \in \mathbb{N},
\]
be a pointed chain in $C$. (Hence $\gamma$ is also a pointed chain in $C'$.)

For each $i \in \{1,2,\dots,t\}$, let $s_i$ be the state of $c_i$ in $C$ and
let $s_i'$ be the state of $c_i$ in $C'$.  
Assume that there exists an index $k \in \{1,2,\dots,t\}$ such that
\[
\mathrm{label}(s_k) \neq \mathrm{label}(s_k').
\]

Then there exists an index $\ell \in \{1,\dots,t\}$ such that the states 
\[
A_{\varphi}(C)(c_\ell)
\quad\text{and}\quad
A_{\varphi}(C')(c_\ell)
\]
have different labels.

(Informally, the map $A_{\varphi}$ is injective when restricted to pointed chains.)
\end{lemma}

\begin{proof}

Suppose that $A_\varphi(C)(c_j)=A_\varphi(C')(c_j')$ for all $j \in \{1, 2, \dots, t\}$. As their labels are computed using the local transition rule from Definition~\ref{def:A_varphi}, we can write it as
\begin{align*}
      \bigl(\mathrm{label}(s_1) & \oplus \mathrm{label}(s_2),\, \dots,\, \mathrm{label}(s_{t-1}) \oplus \mathrm{label}(s_{t}),\, \mathrm{label}(s_{t})\bigr)
      \\ & =
    \bigl(\mathrm{label}(s_1') \oplus \mathrm{label}(s_2'),\, \dots,\, \mathrm{label}(s_{t-1}') \oplus \mathrm{label}(s_{t}'),\, \mathrm{label}(s_{t}')\bigl).  
\end{align*}

From the last coordinate we obtain $\mathrm{label}(s_{t}) = \mathrm{label}(s_{t}')$.  
Next, comparing the $(t-1)$-th coordinates gives
\[
    \mathrm{label}(s_{t-1}) \oplus \mathrm{label}(s_{t}) = \mathrm{label}(s_{t-1}') \oplus \mathrm{label}(s_{t}'),
\]
which implies
\begin{align*}
        \mathrm{label}&(s_{t-1})  = \bigl(\mathrm{label}(s_{t-1})\oplus \mathrm{label}(s_{t})\bigr) \oplus \mathrm{label}(s_{t}) \\ & = \bigl(\mathrm{label}(s_{t-1}')\oplus \mathrm{label}(s_{t}')\bigr) \oplus \mathrm{label}(s_{t}) \\ & = \mathrm{label}(s_{t-1}') \oplus \bigl(\mathrm{label}(s_{t}) \oplus \mathrm{label}(s_{t})\bigr) = \mathrm{label}(s_{t-1}').
\end{align*}

By induction, we conclude
\[
    \mathrm{label}(s_{i})  = \mathrm{label}(s_{i}')  \quad \text{for all } 1 \leq i \leq t.
\]
Hence 
\[
\bigl(\mathrm{label}(s_{1}), \mathrm{label}(s_{2}) , \dots, \mathrm{label}(s_{t}) \bigr) = \bigl(\mathrm{label}(s_{1}'), \mathrm{label}(s_{2}') , \dots, \mathrm{label}(s_{t}') \bigr),
\]
and the statement holds.

\end{proof}

Analogously to pointed chains we define a cycle.

\begin{definition}
    A \emph{cycle} is any sequence of cells $c_1, \dots, c_s \in \Config_{n,m}$, where $s \in \mathbb{N}$, such that:
    \begin{itemize}
        \item $c_i$ is blue for $i \in \{1, \dots, s\}$ and 
        \item $c_{i+1}$ is the successor of $c_{i}$ for $i \in \{1, \dots, s-1\}$, $c_1$ is the successor of $c_s$
    \end{itemize}
\end{definition}

\begin{lemma}\label{lm:at_least_one_red}
    If $C \in \Config_{n,m}$ contains a cycle, then every cell in $C$ belongs to this cycle.
\end{lemma}

\begin{proof}
    Let $\mathcal{Z}$ be a cycle in $C$.
Since successors of cells in a cycle also belong to the cycle, it suffices to show
that from any one cell in $\mathcal{Z}$ we can reach all other cells of the finite
configuration by repeatedly following successor edges.

By the definition of a cycle, $\mathcal{Z}$ contains at least one blue cell.
Since this cell is blue, the coordinate component of its state $s$ is not $\square$.
First assume that $\coord(s)=(n,0)$.

First, we show by backward induction on $i$ that $\mathcal{Z}$ contains all cells
whose coordinate component of the state is $(i,0)$, for $0 \le i \le n$.

Assume that a cell whose coordinate component of the state is $(i,0)$, for some $i \ge 1$, belongs to $\mathcal{Z}$.
Then, by Definition~\ref{def:successor}, the successor of this cell
has coordinate component $(i-1,0)$.
Since the successor of a cell in a cycle also belongs to the cycle,
there is a cell in $\mathcal{Z}$ whose coordinate component of the state
is $(i-1,0)$.

Similarly, by induction on $j$, we show that $\mathcal{Z}$ contains all cells
whose coordinate component of the state is $(0,j)$ for $0 \le j \le m+1$.

We now prove by induction on $i \le n$ that all cells whose coordinate
component of the state is $(i,j)$ belong to $\mathcal{Z}$.

The case $i=0$ was already shown.
Assume that for some $k < n$, all cells whose coordinate component
of the state is $(i,j)$, where $i \le k$ and $j \le m+1$, belong to $\mathcal{Z}$.
We consider two cases.

\begin{itemize}
    \item If $k$ is even, then the successor of the cell whose coordinate
    component of the state is $(k,m+1)$ has coordinate component
    $(k+1,m+1)$, and hence this cell belongs to $\mathcal{Z}$.
    The proof that all cells whose coordinate component of the state
    is $(k+1,j)$ for $1 \le j \le m$ belong to $\mathcal{Z}$
    is analogous to the argument used above.

    \item If $k$ is odd, then the successor of the cell whose coordinate
    component of the state is $(k,1)$ has coordinate component
    $(k+1,1)$.
    Hence this cell belongs to $\mathcal{Z}$.
    Again, by repeatedly following successor edges, we obtain that all cells
    whose coordinate component of the state is $(k+1,j)$,
    for $1 \le j \le m$, belong to $\mathcal{Z}$.
\end{itemize}

In both cases, all cells whose coordinate component of the state
is in row $k+1$ belong to $\mathcal{Z}$, which completes the induction.

We have proved that if a cycle contains a cell whose coordinate
component is $(n,0)$, then it must contain all cells of the finite
configuration $C \in \Config_{n,m}$.

The choice of $(n,0)$ was only for convenience.
The same argument applies to any cell whose coordinate component
is $(i,j) \in [n+1]\times[m+2]$.
Starting from such a cell, the local correctness rules force the
presence of all its neighbours, and by repeating this argument,
the whole rectangle $[n+1]\times[m+2]$ must belong to the cycle.

Hence every cycle that contains at least one cell of $C$
contains all cells of $C$.
\end{proof}

An easy consequence of the previous lemma is that if a configuration contains at least one red cell, then it does not contain a cycle.

\begin{theorem}\label{thm:red_cell}
    Fix $C \in \Config_{n,m}$ and let $U = \{C' \in \Config_{n,m} \mid C' \sim C\}$. The map $A_\varphi: U \to U$ is injective if and only if $C$ contains a red cell.
\end{theorem}
\begin{proof}
The left to right direction was proven in Lemma~\ref{lm:blue_table}. 

Assume than $C$ contains at least one red cell. There are two options: either all the cells in $C$ are red, or there is at least one blue cell. 

In the first case   $A_\varphi: U \to U$ is an identity mapping and it is trivially injective.
 
In the second case consider a blue cell $c$. We want to show that it necessarily belongs to a pointed chain. As it is blue, it has a successor. Either the successor is red and we are done, or it is blue. $C$ is finite and according to Lemma~\ref{lm:at_least_one_red},  $c$ is not in a cycle. Therefore iterating successors we reach a red cell or a quiescent cell. However, the last case contradicts the local correctness rules. Hence, $c$ belongs to a pointed chain. 
Therefore, $C$ is a  union of pointed chains. 

Assume that for $C_1,C_2 \in U$, it holds $A_\varphi (C_1) = A_\varphi (C_2)$, but $C_1 \neq C_2$. 
Then there must exist a cell that appears in both $C_1$ and $C_2$ but has different labels in the two configurations. 
Since this cell belongs to a pointed chain, this contradicts the injectivity of $A_\varphi$ on pointed chains (Lemma~\ref{lm:global_injective}). Therefore, $A_\varphi: U\to U$ is injective. 
\end{proof}

We note that the next result is a version of Durand’s theorem (see~\cite{DURAND1994387}). In simple words, even if we only look at configurations of bounded size, it is still "hard" to decide whether a given cellular automaton is injective on those configurations or not.

\begin{theorem}\label{thm:injective_not_sat}
     $A_\varphi: \Config_{n,m} \to \Config_{n,m}$ is injective if and only if $\varphi$ is not satisfiable. 
\end{theorem}
\begin{proof}
We prove both implications.

($\Rightarrow$) Assume that $\varphi$ is satisfiable. Then there exists an assignment $\tuple b$ such that $\varphi(\tuple b)=1$.

By Lemma~\ref{lm:table_loc_correct}, the configuration $\Table_\varphi(\tuple b)$ is locally correct. By the definition of $\Table_\varphi$, its output cell contains the value $\varphi(\tuple b)=1$.

By Lemma~\ref{lm:blue_table}, the restriction of $A_\varphi$ to the set of configurations similar to $\Table_\varphi(\tuple b)$ is not injective. Therefore, the map
\[
A_\varphi : \Config_{n,m} \to \Config_{n,m}
\]
is not injective.

($\Leftarrow$) Now assume that $\varphi$ is not satisfiable. This means that for each assignment $\tuple a$ we have $\varphi(\tuple a) = 0$. Fix an arbitrary
$\tuple{a}$ and consider a configuration $C$ which has $\tuple a$ in the $0$th row. 

If $C$ is locally correct, then according to Lemma~\ref{lm:new_lemma} $C$ is similar to $\Table_\varphi(\tuple a)$ and therefore the value in the output cell is $\varphi(\tuple a) = 0$. If $C$ is not locally correct, this means that some cell does not satisfy local correctness rules.

In both cases according to Definition~\ref{def:blue_red} $C$ contains a red cell.  Therefore, by
Theorem~\ref{thm:red_cell}, for every configuration $C$ the restriction
\[
A: U \to U,
\quad
U = \{C' \in \Config_{n,m} \mid C' \sim C\},
\]
is injective. Moreover, for any $C'' \in \Config_{n,m}$, if $C'' \not\sim C$, letting
\[
V = \{C' \in \Config_{n,m} \mid C' \sim C''\},
\]
the sets $U$ and $V$, as well as $A_\varphi(U)$ and $A_\varphi(V)$ are disjoint. Hence
\[
A_\varphi : \Config_{n,m} \to \Config_{n,m}
\]
is injective.
\end{proof}

\section{Formalization in bounded arithmetic}
Our next goal is to formalize the proof of the if-then direction (left to right) of Theorem~\ref{thm:injective_not_sat} within the theory~$V^0$. We begin by formalizing finite configurations and cellular automata in this theory.

\subsection{Additional coding constructions}
We now explain how to code the cellular automaton $A_\varphi$ inside $V^0$ and the table $\Table_\varphi(\tuple a)$.  
All objects are finite and bounded by a polynomial in $|\varphi|$, therefore they can be represented as finite sets of numbers.

\begin{remark}
     For natural numbers $a,b$, where $a\leq b$ we use the abbreviations
\[
\forall i \in \{a, \dots, b\} \; \psi(i)
\]
is an abbreviation for the bounded formula
\[
\forall i\leq b \; (a \leq i \rightarrow \psi(i))
\]
and the notation
\[
\exists i \in \{a, \dots, b\} \; \psi(i)
\]
is an abbreviation for the bounded formula
\[
\exists i\leq b \; (a \leq i \wedge \psi(i)).
\]

\end{remark}

\subsubsection{Coding of $A_\varphi$}\label{subsub:coding_af}

The automaton $A_\varphi=(2,S_{n,m},N,f_\varphi)$ is coded by defining $S_{n,m}$ and $f_\varphi$ as bounded sets.

\textbf{Coding of $S_{n,m}$}

In the coding of $S_{n,m}$ we use the notation
\[
s \in ([n+1] \cup \{\square\}) \times ([m+2] \cup \{\square\})
\times \{-1,0,1,\square\} \times \{0,1,\square\}^3 \times \{0,1,\square\}
\]
as a shorthand.

Formally, this means that every component of a state $s$
satisfies the corresponding bounded condition. For example,
\[
i \leq n \ \text{or}\ i=\square,
\qquad
j \leq m+1\ \text{or}\ j=\square,
\]
\[
\flag \in \{-1,0,1\} \ \text{or}\ \flag=\square,
\]
and similarly for the remaining components.

\begin{lemma}
The theory $V^0$ proves that for every CNF formula $\varphi$
there exists a finite set of states $S_{n,m}$
as defined in Equation~\eqref{eq:set_of_states_bound}.
\end{lemma}
\begin{proof}
Fix a CNF formula $\varphi$ with $n$ clauses and $m$ propositional variables.
We define the set $S_{n,m}$ by a $\Sigma^B_0$-formula and then apply
the bounded comprehension axiom.

First, we fix a polynomial bound $t(n,m)$ large enough to encode all components
$(i,j,\flag,a,\pd,\pc,\mathrm{label})$ 
as a single number (using the standard coding of tuples described in
Subsection~\ref{coding_relations_functions}).  Such $t(n,m)$ exists since all the tuples are from $([n+1] \cup \{\square\}) \times([m+2]\cup \{\square\})\times\{-1, 0, 1, \square\} \times \{0,1, \square\}^3 \times \{0,1\}$, which has polynomial size in $n$, $m$ (therefore in $|\varphi|$).

We now define a bounded formula $\psi(s)$ saying that the code $s$
represents a state of one of the seven types.
Formally, for
\begin{align*}
((i,j),&\flag,a,\pd,\pc,\mathrm{label})
\in \\& ([n+1]\cup\{\square\})
\times([m+2]\cup\{\square\})
\times\{-1,0,1,\square\}
\times\{0,1,\square\}^3
\times\{0,1\},
\end{align*}

we put
\[
\psi(\varphi,(i,j),\flag,a,\pd,\pc,\mathrm{label})
\]
to be the disjunction of the following seven bounded conditions:

    \begin{enumerate}[label={(\arabic*)}]
        \item $(i, j)=(0,0) \wedge \flag =a=\pd=\pc =\square$ \hfill \text{// top-left cell}
        \item $i \in \{1, 2, \dots, n \} \wedge j=0 \wedge \flag =a=\pd=\pc =\square$ \hfill \text{// $0$th column}
        \item $i =0 \wedge j \in \{1, 2, \dots, m \}  \wedge a \in \{0, 1\} \wedge  \flag =\pd=\pc =\square$ \hfill \text{// $0$th row}
        \item $(i,j) =(0, m+1) \wedge  \flag =a=\pd=\pc =\square$ \hfill\text{// top-right cell}
        \item $i \in \{1, 2, \dots, n\} \wedge j= m+1 \wedge \mathsf{pc} \in \{0, 1\} \wedge \flag =a=\pd =\square $\hfill \text{// last column}
        \item $ (i,j) \in \{1, \dots, n\} \times \{1, \dots, m\} \wedge \mathrm{flag}\in \{-1, 0, 1\}  \wedge a \in \{0, 1\} \wedge \mathsf{pd} \in \{0, 1\} \wedge \pc =\square$\hfill \text{// main body}
        \item $i = j = \flag = a=\pd=\pc = \mathrm{label} = \square$\hfill \text{// quiescent state}
    \end{enumerate}

Each of these clauses uses only bounded number quantifiers and bounded
case distinctions over finite sets.
Hence $\psi$ is a $\Sigma^B_0$-formula.

By the bounded comprehension scheme $\Sigma^B_0\text{-COMP}$,
$V^0$ proves the existence of a set
\[
S_{n,m} \le t(n,m)
\]
such that
\[
\forall s < t(n,m) \;
\bigl( S_{n,m}(s) \leftrightarrow \psi(s) \bigr).
\]
Therefore $V^0$ proves that the set $S_{n,m}$ exists.
\end{proof}

\medskip

\textbf{Coding of $f_\varphi$}

\begin{lemma}
The theory $V^0$ proves that for every CNF formula $\varphi$
there exists a transition function $f_\varphi$
as defined in Definition~\ref{def:A_varphi}.
\end{lemma}
\begin{proof}
The proof is similar to the proof for $S_{n,m}$.
We define the graph of $f_\varphi$ by a $\Sigma^B_0$-formula
and then apply bounded comprehension.

If $t(n,m)$ is the polynomial bound used for $S_{n,m}$,
then the size of $S_{n,m}^6$ is bounded by $p(n,m) = (t(n,m))^6$, which is polynomial in $|\varphi|$.

For each tuple $(s_1,s_2,s_3,s_4,s_5,s_1') \in S_{n,m}^6$
we define:
\begin{align*}
(s_1, s_2, s_3, s_4, s_5, s_1') \in f_\varphi \iff \theta(\varphi, s_1, s_2, s_3, s_4, s_5, s_1'),
\end{align*}

where $\theta(\varphi, s_1, s_2, s_3, s_4, s_5, s_1') = \bigl((1) \wedge (2) \rightarrow (3)\bigr) \wedge \bigl(\neg((1) \wedge (2)) \rightarrow (4)\bigr)$ with (1)--(4) denoting the following formulas:

\begin{enumerate}[label={(\arabic*)}]
    \item is the conjunction of all conditions from  Definition~\ref{def:locally_correct} saying that the cell with state $s_1$ is locally correct.

For example, the index-consistency
rule (A) for $i\in\{2,\dots,n-1\}$ and $j\in\{2,\dots,m\}$ is formalized as
\begin{align*}
    \coord(s_1)=(i,j) \iff &\coord(s_2) = (i, j+1) \wedge \coord(s_3) = (i, j-1) \\& \wedge \coord(s_4) = (i+1, j) \wedge \coord(s_5) = (i-1, j).
\end{align*}

Similarly, the vertical consistency rule (C2) for $i\ge2$ is formalized as
\begin{align*}
    \coord(s_1) = (i,m+1) \rightarrow \bigl( \pc(s_1) =1 \iff \pc(s_5) = 1 \wedge \pd(s_3) = 1 \bigr).
\end{align*}

\item says that if the cell with state $s_1$ is output cell, then the value of partial conjunction in its state is $1$:
\begin{align*}
    \coord(s_1) = (n, m+1) \rightarrow \pc(s_1) = 1.
\end{align*}

\item is the local transition rule if the cell with $s_1$ is blue. All components of the state remain unchanged, except for the label. 
Formally, we require:
\begin{align*}
    \coord(s_1') &= \coord(s_1),\\
    \flag(s_1') &= \flag(s_1),\\
    a(s_1') &= a(s_1),\\
    \pd(s_1') &= \pd(s_1),\\
    \pc(s_1') &= \pc(s_1),
\end{align*}
and the label is updated by
\begin{align*}
    \forall s \in \{s_2, s_3, s_4, s_5 \} \; \bigl(\coord(s)= & \mathrm{suc}(\coord(s_1)) \\& \rightarrow \mathrm{label}(s_1') 
    = \mathrm{label}(s_1) \oplus \mathrm{label}(s)\bigr),
\end{align*}

\item says that  the state does not change, i.e. it is the local transition rule if the cell with $s_1$ is red:
\[
s_1' = s_1.
\]
\end{enumerate}

By the bounded comprehension scheme $\Sigma^B_0\text{-COMP}$,
the theory $V^0$ proves the existence of a set
\[
f_\varphi \le p(n,m)
\]
such that
\[
\forall s < p(n,m)\;
\bigl(f_\varphi(s) \leftrightarrow \theta(\varphi, s)\bigr),
\]
where $\theta(\varphi, s)$ is a $\Sigma^B_0$-formula.
\end{proof}

Therefore both $S_{n,m}$ and $f_\varphi$ are $\Sigma^B_0$-definable, and hence $A_\varphi$ exists in $V^0$.

\subsubsection{Coding of $\Table_\varphi(\tuple a)$}\label{subsub:coding_table_af}

Recall that a configuration is a function $C : \mathbb{Z}^2 \to S_{n,m}$. Since we work only with bounded configurations inside $[n+1]\times[m+2]$, the configuration $C \in \Config_{n,m}$ can be represented by its graph:
\[
C \subseteq [n+1]\times[m+2]\times S_{n,m}.
\]

Recall that a configuration $C \in \Config_{n,m}$ is equal to
$\Table_\varphi(\tuple a)$ if and only if the following two conditions hold:
\begin{enumerate}
    \item $C$ is correct according to Equation~\ref{eq: definition_table},
    \item the labels in each cell's state are equal to $0$.
\end{enumerate}

These conditions can be formalized as a $\Sigma^B_0$-formula, 
allowing us to define $\Table_\varphi(\tuple a)$ within $V^0$.

\begin{lemma}\label{lm:V_0_table}
The theory $V^0$ proves that for every CNF formula $\varphi$ with $m$ propositional variables and for every assignment $\tuple a \in \{0,1\}^m$, there exists $\Table_\varphi(\tuple a)$.
\end{lemma}
\begin{proof}
    
For each $(i,j,s) \in [n+1] \times [m+2] \times S_{n,m}$ we define
\[
 \bigl((i, j, s) \in \Table_\varphi(\tuple a) \iff \Psi(\varphi, \tuple a, i,j, s) \bigr),
\]
where $\Psi(\varphi, \tuple a, i,j, s) = (a) \wedge (b) \wedge (c)$, where
\begin{enumerate}[label=(\alph*)]

\item is a conjunction of conditions (A)--(E) saying that $(i, j, s)$ is correct.
\begin{enumerate}[label=(\Alph*)]
    \item true coordinates of the cell should be equal to the ones in its state
    \begin{align*}
    \coord(s)=(i,j),        
    \end{align*}

    \item the flags written in the cells of the main body of $C$ indeed encode $\varphi$:
    \begin{align*}
    (i,j) \in \{1,\dots,n\} \times \{1,\dots,m\}\; \forall k \in&\{-1, 0, 1\} 
     \rightarrow \\& 
    (\flag(s)=k \;\leftrightarrow\; (i,j,k)\in\varphi),
    \end{align*}
    
    \item partial disjunctions have correct values
    \begin{align*}
    (i,j) \in & \{1,2 \dots, n\} \times \{1, 2 \dots, m\} \rightarrow \bigr(\pd(s) =1 \leftrightarrow \exists \tilde{j}\in\{1, 2, \dots, j\} \\& [(i, \tilde{j}, 1) \in \varphi \wedge a_{\tilde{j}}=1] \vee [(i, \tilde{j}, -1) \in \varphi \wedge a_{\tilde{j}}=0] \bigl),        
    \end{align*}

        \item partial conjunctions have correct values
    \begin{align*}
    (i,j) & \in \{1,2 \dots, n\} \times \{m+1\} \rightarrow   \bigr(\pc(s) =1 \leftrightarrow \forall \tilde{i} \in \{1, 2, \dots, i\} \; \\& \exists \tilde{j}\in\{1, 2, \dots, m\}  [(\tilde{i}, \tilde{j}, 1) \in \varphi \wedge a_{\tilde{j}}=1] \vee [(\tilde{i}, \tilde{j}, -1) \in \varphi \wedge a_{\tilde{j}}=0]\bigl),        
    \end{align*}

    \item a conjunction of technical conditions (i)--(iv)

    \begin{enumerate}
        \item $\bigl(j =0 \vee (i,j)=(0, m+1)\bigr) \rightarrow \flag(s) = a(s) = \pd(s) = \pc(s) = \square$
        \item $\bigl(i =0 \wedge j \in\{1, 2, \dots, m\}\bigr) \rightarrow \flag(s) = \pd(s) = \pc(s) = \square$
        \item $\bigl(i \in \{1, 2 \dots, n\} \wedge j=m+1\bigr) \rightarrow \flag(s) = a(s) = \pd(s) = \square$
        \item $\bigl(i \in \{1, 2, \dots, n\} \wedge j \in\{1,2, \dots m\}\bigr) \rightarrow \pc(s) = \square$
    \end{enumerate}
\end{enumerate}

    \item is the condition saying that the assignment is indeed $\tuple a = (a_1, a_2, \dots, a_m)$, i.e. the cells in positions $(i,j)$ for $i \in [n+1], j \in \{1, 2, \dots, m\}$ should contain values $a_1, a_2, \dots, a_m$: 
    \[
     (i,j) \in [n+1]\times \{1, 2, \dots, m\} \rightarrow a(s) = a_j,
    \]
    
    \item is the condition that labels in all cells in the finite part of $C$ are $0$ (without this condition the formula would encode any configuration extending $\Table_\varphi(\tuple a)$):
    \[
    \mathrm{label}(s) = 0. 
    \]
\end{enumerate}

Here the bound for the size of $\Table_\varphi(\tuple a)$ is $O(n m t(n,m))$, which is again polynomial in $|\varphi|$.  

By the bounded comprehension scheme $\Sigma^B_0\text{-COMP}$, the theory $V^0$ proves the existence of $\Table_\varphi(\tuple a)$.
\end{proof}

\begin{theorem}\label{thm:V0_injective_unsat}
    $V^0$ proves that if $A_\varphi$ is injective on $\Config_{n,m}$, then $\varphi$ is not satisfiable.
\end{theorem}
\begin{proof}

Assume that $\varphi$ is satisfiable.
Then there exists $\tuple a \in \{0,1\}^m$ such that $\varphi(\tuple a)=1$. The predicate $\mathrm{Sat}(n,m,\tuple a,\varphi)$ is a $\Delta_0$-formula
(see Subsection~\ref{subsub:coding_cnfs}),
hence this assumption is expressible in $V^0$.

    We have shown in Subsections~\ref{subsub:coding_af} and \ref{subsub:coding_table_af} that $V^0$ proves the existence of $f_\varphi$ and $\Table_\varphi(\tuple a)$. Similarly to $\Sigma^B_0$-definition of  $\Table_\varphi(\tuple a)$ in Lemma~\ref{lm:V_0_table}, we can define by $\Sigma^B_0$-formula a configuration $C\sim \Table_\varphi(\tuple a)$ such that the label of each cell's state is $1$. Therefore $V^0$ also proves the existence of such $C$.

In Lemma~\ref{lm:table_loc_correct} we proved that $\Table_\varphi(\tuple a)$ is locally correct. It remains to notice that this proof can be formalized in $V^0$: indeed, all arguments use only bounded quantification over $i\le n$ and $j\le m$ and simple identities defining $\pd$ and $\pc$ from neighbouring cells. The equalities
\[
\bigvee_{u\le j} C_{i,u}
=
\bigvee_{u\le j-1} C_{i,u} \;\vee\; C_{i,j}
\quad\text{and}\quad
\bigwedge_{v\le i} C_v
=
\bigwedge_{v\le i-1} C_v \;\wedge\; C_i
\]
are straightforward bounded calculations. Hence the statement that every cell of $\Table_\varphi(\tuple a)$ satisfies the local correctness rules is expressible by a $\Sigma^B_0$-formula and is provable in $V^0$.

    By definition of $f_\varphi$, the automaton never changes
any component of a state except possibly the label.
Hence it is enough to compare the labels after one step.

Let $c_1$ be any cell of $\Table_\varphi(\tuple a)$
with state $s_1$, and let $s_2,s_3,s_4,s_5$
be the states of its von Neumann neighbours.
Let $s$ be such that
\[
(s_1,s_2,s_3,s_4,s_5,s) \in f_\varphi.
\]

Since $\Table_\varphi(\tuple a)$ is locally correct,
the update rule for labels applies in the ``correct'' case.
By definition of $f_\varphi$,
\[
\mathrm{label}(s)
=
\mathrm{label}(s_1)
\oplus
\mathrm{label}(s_k)
\]
for some $k \in \{2,3,4,5\}$.

In $\Table_\varphi(\tuple a)$ all labels are $0$,
hence
\[
\mathrm{label}(s)=0\oplus 0=0.
\]
Thus
\[
A_\varphi\bigl(\Table_\varphi(\tuple a)\bigr)
=
\Table_\varphi(\tuple a).
\]

    By an analogous argument, we can show that the labels of all cells in image of $C$ are $0$.

    Therefore images of $\Table_\varphi(\tuple a)$ and $C$ are equal (to $\Table_\varphi(\tuple a)$), and thus $A_\varphi$ is not injective.
\end{proof}

\section{Inverse automata}\label{sec:inverse_automata}

We have a CNF formula $\varphi$ which determines parameters $n$ and $m$ and the automaton $A_\varphi =(2, S_{n,m}, N, f_\varphi)$ we defined in the previous sections.  
For simplicity, in this section we abbreviate the notation as follows:
\[
\Config := \Config_{n,m}, 
\qquad 
A := A_\varphi,
\qquad
S := S_{n,m}, 
\qquad 
f := f_\varphi.
\]

Assume that $B$ is a cellular automaton defined by
\[
B = (2, M, S, g),
\]
where $M$ is a neighbourhood of size $\mu \ge 1$, and 
\[
g : S^\mu \to S
\]
is the transition function. Again we study the action of $B$ on $\Config$ only.

For a cell $c \in [n+1] \times [m+2]$, we denote by $N(c)$ the tuple of
cells in the von Neumann neighbourhood of $c$, in the same order as in
the definition of the neighbourhood. That is,
\[
N(c) = (c,\, c+(0,1),\, c+(0,-1),\, c+(1,0),\, c+(-1,0)).
\]
Similarly, $M(c)$ denotes the tuple of cells in the $M$-neighbourhood of $c$. 

When we apply set operations to these neighbourhoods, we implicitly
identify the tuple with the set of its elements. In particular, we define
\[
N(M(c)) \; := \; \bigcup_{d \in M(c)} N(d).
\]

Hence we may consider the composition
\[
B \circ A : \Config \to \Config , \qquad C \mapsto B(A(C)).
\]

Our goal is to express by a $\Sigma^B_0$-formula that $B$ is an inverse of $A$.
More precisely, we want to find a $\Sigma^B_0$ formula equivalent to the statement that for every configuration
$C \in \Config$ we have
\[
B(A(C)) = C .
\]

\begin{lemma}
Let $c \in [n+1] \times [m+2]$ and let $C \in \Config$. The state of $c$ in the
configuration $B(A(C))$ depends only on the states of the cells from
$N(M(c))$ in $C$.
\end{lemma}

\begin{proof}
    Let $s$ be the state of $c$ in $B(A(C))$. By definition, 
    \[
    s = g(s_1, s_2, \dots, s_\mu),
    \]
   where $s_1, s_2, \dots, s_\mu$ are the respective states of the cells
    $c_1, c_2, \dots, c_\mu$ in $A(C)$, where $M(c) = (c_1, c_2, \dots, c_\mu)$.
    
    For each $i \in \{1, 2, \dots, \mu\}$, we have
    \[
    s_i = f(s^i_1, s^i_2, s^i_3, s^i_4, s^i_5),
    \]
    where $s^i_1, s^i_2, s^i_3, s^i_4, s^i_5$ are the states of von Neumann neighbours of $c_i$ in $C$. 
    
    Substituting these expressions into the definition of $s$ gives
    \[
    s = g\bigl(f(s^1_1, s^1_2, s^1_3, s^1_4, s^1_5), f(s^2_1, s^2_2, s^2_3, s^2_4, s^2_5), \dots, f(s^\mu_1, s^\mu_2, s^\mu_3, s^\mu_4, s^\mu_5)\bigr).
    \]

Thus $s$ depends only on the states of the neighbours of the cells
$c_1,\dots,c_\mu$ in $C$. Since the von Neumann neighbours of $c_i$ are exactly
$N(c_i)$, these cells form the set
\[
N(c_1) \cup N(c_2) \cup \dots \cup N(c_\mu) = N(M(c)).
\]
\end{proof}

\begin{lemma}\label{lm:invertible_condition}
    Let $n,m, S, \Config, A, B$ be as above. The following statements are equivalent:
    \begin{enumerate}[label=(\roman*)]
        \item $B$ is the inverse of $A$ on $\Config$, i.e.
    \[
    \forall C \in \Config \; B(A(C)) = C .
    \]

        \item For every cell $c \in [n+1]\times[m+2]$, the following holds.
Let $M(c) = (c_1,c_2,\dots,c_\mu)$. For each $i \in \{1,2,\dots,\mu\}$ let
$s^i_1,s^i_2,s^i_3,s^i_4,s^i_5 \in S$ be arbitrary states of the cells in
$N(c_i)$, where $s^1_1$ is the state of $c$.
Then
\[
s^1_1 =
g\bigl(
f(s^1_1,s^1_2,s^1_3,s^1_4,s^1_5),
f(s^2_1,s^2_2,s^2_3,s^2_4,s^2_5),
\dots,
f(s^\mu_1,s^\mu_2,s^\mu_3,s^\mu_4,s^\mu_5)
\bigr).
\]
    \end{enumerate}
\end{lemma}
\begin{proof}
Assume first that $B$ is the inverse of $A$ on $\Config$. Let $C \in \Config$
and let $c$ be a cell of $C$ with state $s^1_1$. Since $B(A(C))=C$, the state
of $c$ in $B(A(C))$ is also $s^1_1$.

By the previous lemma, the state of $c$ in $B(A(C))$ is
\[
g\bigl(
f(s^1_1,s^1_2,s^1_3,s^1_4,s^1_5),
f(s^2_1,s^2_2,s^2_3,s^2_4,s^2_5),
\dots,
f(s^\mu_1,s^\mu_2,s^\mu_3,s^\mu_4,s^\mu_5)
\bigr),
\]
where for each $i \in \{1,2,\dots,\mu\}$ the values
$s^i_1,s^i_2,s^i_3,s^i_4,s^i_5$ are the states of the cells in $N(c_i)$ in $C$
and $M(c)=(c_1,c_2,\dots,c_\mu)$. Since this state equals the state of $c$ in $C$,
we obtain the required equality.

The converse direction follows immediately from the same expression for the
state of $c$ in $B(A(C))$ given by the previous lemma.
\end{proof}

\subsection{Coding states in neighbourhood $M$}
We already introduced the coding of finite sequences in Subsection~\ref{coding_relations_functions}.  Using this coding we can encode
the neighbourhood $N(c)$ of a cell $c$. However, the same method cannot be
used directly for the neighbourhood $M(c)$, since the size of $M(c)$ is not some fixed constant, but may vary with $M$.
Therefore we need a way to encode sequences whose length is not fixed. 

This is well-known in bounded arithmetic, cf. \cite[Lemma~9.3.2]{krajicek2019}. Informally the idea of such a coding can be explained as follows.

Assume we have a sequence
\[
(s_1,s_2,\dots,s_\mu) \in S^\mu .
\]
We encode it by a number $s$ whose bit in position $\langle i,j\rangle$ is the $i$-th bit of $s_j$. This can be done with $s$ of length 
less than $10 \mu \log S$, i.e.:
\[
s  < |S|^{10\mu}.
\]

The difficulty is that in $V^0$ we cannot directly use the bound
$|S|^{10\mu}$ when $\mu$ is a variable. To avoid this problem we introduce an
additional parameter $t$ that serves as an upper bound for all such codes.

We define a $\Sigma^B_0$ formula
\[
\Inv(x,y,A,B,t)
\]
to express that the automaton is locally invertible with respect to this
bound. Formally,
\[
\Inv(x,y,A,B,t) \;:=\;
|S|^{10\mu} < t \;\wedge\; (ii),
\]
where $(ii)$ is the condition from Lemma~\ref{lm:invertible_condition}, written using the coding above.

\medskip
\noindent\fbox{%
\parbox{0.95\linewidth}{
\textbf{Condition (ii).} 
Condition from Lemma~\ref{lm:invertible_condition}, modified for $V^0$.

For all numbers $s_1,s_2,s_3,s_4,s_5 < t$ satisfying
\begin{align*}
s_j &= (s^1_j,s^2_j,\dots,s^\mu_j) \qquad \text{for } j=1,\dots,5,\\
s^1_1 &\text{ is the state of } c \text{ in } C,\\
(s^1_1,s^2_1,\dots,s^\mu_1) &\text{ are the states of the cells in } M(c),\\
(s^i_1,s^i_2,s^i_3,s^i_4,s^i_5) &\text{ are the states of the cells in }
N(c_i)
\end{align*}
for every $i \in \{1,\dots,\mu\}$, where
\[
M(c)=(c_1,c_2,\dots,c_\mu),
\]
the following equality holds:
\[
s^1_1 =
g\bigl(
f(s^1_1,s^1_2,s^1_3,s^1_4,s^1_5),
f(s^2_1,s^2_2,s^2_3,s^2_4,s^2_5),
\dots,
f(s^\mu_1,s^\mu_2,s^\mu_3,s^\mu_4,s^\mu_5)
\bigr).
\]
}
}
\medskip

\subsection{Invertible automata are injective in $V^0$}

\begin{lemma}\label{lm:invertible_implies_injective}
\[
V^0 \vdash
\Inv(x,y,A,B,t)
\;\rightarrow\;
\bigl[
\forall C,C' \in \Config\;
(A(C)=A(C') \rightarrow C=C')
\bigr].
\]
\end{lemma}
\begin{proof}
Assume $\Inv(x,y,A,B,t)$ and let $C,C' \in \Config$ satisfy
\[
A(C)=A(C').
\]

Recall that we defined configuration as a map 
$C : \mathbb{Z}^2 \to S$. Therefore in order to prove $C=C'$ it suffices to show that for every cell
$c \in [n+1]\times[m+2]$ we have
\[
C(c)=C'(c).
\]

Fix such a cell $c$ and write
\[
M(c)=(c_1,\dots,c_\mu).
\]

For each $i\le \mu$ let $s^i_j$ be the states of the cells in $N(c_i)$ in $C$, and let $t^i_j$ be the states of the cells in $N(c_i)$ in $C'$.

By definition of $A$ we have
\[
A(C)(c_i)=f(s^i_1,s^i_2,s^i_3,s^i_4,s^i_5)
\]
and
\[
A(C')(c_i)=f(t^i_1,t^i_2,t^i_3,t^i_4,t^i_5).
\]

Since $A(C)=A(C')$, it follows that for every $i\le\mu$,
\[
f(s^i_1,s^i_2,s^i_3,s^i_4,s^i_5)
=
f(t^i_1,t^i_2,t^i_3,t^i_4,t^i_5).
\]

Now apply condition (ii) from $\Inv(x,y,A,B,t)$ to the
configuration $C$. We obtain
\[
C(c)=
g\bigl(
f(s^1_1,s^1_2,s^1_3,s^1_4,s^1_5),
\dots,
f(s^\mu_1,s^\mu_2,s^\mu_3,s^\mu_4,s^\mu_5)
\bigr).
\]

Applying the same condition to $C'$ gives
\[
C'(c)=
g\bigl(
f(t^1_1,t^1_2,t^1_3,t^1_4,t^1_5),
\dots,
f(t^\mu_1,t^\mu_2,t^\mu_3,t^\mu_4,t^\mu_5)
\bigr).
\]

Since the corresponding arguments of $g$ are equal, the two
expressions are equal. Hence
\[
C(c)=C'(c).
\]

Since $c$ was arbitrary, we conclude
\[
C=C'.
\]
\end{proof}

\section{Lower bound for inverse automata}\label{sec:lower_bound}

Following the idea in Chapter~$4$ of~\cite{Cavagnetto} we shall interpret the inverse automaton to $A_\varphi$ as a propositional refutation of $\varphi$.

\begin{definition}[$P_{\mathrm{CA}}$]
We define a propositional proof system $P_{\mathrm{CA}}$ as follows.

Let $\varphi$ be a CNF formula. A refutation of $\varphi$ in $P_{\mathrm{CA}}$ is a cellular automaton
\[
B = (2,S,M,g)
\]
such that:
\begin{enumerate}
    \item $S = S_{n,m}$ is the set of states determined by $\varphi$ as in Section~\ref{sec:inverse_automata};
    \item if $t := |B|$ and  $\mu$ is the number of cells in $M$, then $\Inv(n,m,A_\varphi,B,t)$ holds.
\end{enumerate}
\end{definition}

\begin{lemma}
   $P_{\mathrm{CA}}$  is a sound and complete proof system.
\end{lemma}
\begin{proof}
    The provability predicate is $\Sigma^B_0$ and hence p-time decidable.

    Suppose $\Inv(n,m,A_\varphi,B,t)$ holds.
    By Lemma~\ref{lm:invertible_implies_injective} and modus ponens,
    this implies that $A_\varphi$ is injective.
    By Theorem~\ref{thm:injective_not_sat}, $A_\varphi$ is injective
    if and only if $\varphi$ is unsatisfiable.
    Therefore $P_{\mathrm{CA}}$ is sound.

For completeness, assume that the neighbourhood $M$ is large enough so that
for every cell $c \in[n+1]\times[m+2]$ the set $M(c)$ contains the whole
region $[n+1]\times[m+2]$. Thus the function $g$ "sees" the whole configuration
$C$.

If $A_\varphi$ is injective, then for every configuration $C \in A_\varphi(\Config)$ there exists
exactly one configuration $C'\in \Config$ such that
\[
A_\varphi(C') = C .
\]
We take this configuration $C'$ and define $g$ so that,
when it is applied to the states of the cells in $M(c)$ as they
appear in the configuration $C$, it outputs the state of the
cell $c$ in $C'$.

Thus $g$ reconstructs the predecessor configuration cell by cell,
so it is the inverse of $A_\varphi$. Hence $\Inv(n,m,A_\varphi,B,t)$ holds. Therefore $P_{\mathrm{CA}}$ is complete.
\end{proof}

\begin{lemma}\label{lm:V0_Q_sound}
$V^0$ proves that $P_{\mathrm{CA}}$ is sound.
    \[
    \mathrm{Refut}_{P_\mathrm{CA}}(x,y,t,B,\varphi) \longrightarrow \varphi \notin \SAT
    \]
\end{lemma}
\begin{proof}
Assume $ \mathrm{Refut}_{P_\mathrm{CA}}(n,m,t,B,\varphi)$, where $x:=n$ and $y:= m$ are determined by $\varphi$.  
By definition of $ \mathrm{Refut}_{P_\mathrm{CA}}$ (see Subsection~\ref{reflection_principle}), this implies that
$\Inv(n,m,A_\varphi,B,t)$ holds. By Lemma~\ref{lm:invertible_implies_injective}, $V^0$ proves that if
$\Inv(n,m,A_\varphi,B,t)$ holds, then the cellular automaton $A_\varphi$
is injective. By Theorem~\ref{thm:V0_injective_unsat}, $V^0$ proves that if
$A_\varphi$ is injective, then the formula $\varphi$ is unsatisfiable.
Therefore $\varphi \notin \SAT$.
\end{proof}

Recall the formula $\lnot \mathrm{ontoPHP}_k$ from Example~\ref{ex:pigeonhole}. Note that $\lnot \mathrm{ontoPHP}_k$ has $k(k+1)$ variables and $O(k^3)$ clauses.

\begin{theorem}\label{thm:main_lower_bound}
There exists a constant $\epsilon > 0$ such that the following holds for all sufficiently large $k$.
Let 
\[
\varphi := \lnot \mathrm{ontoPHP}_k,
\]
and let $B=(2,S,M,g)$ be a cellular automaton that is inverse to $A_\varphi$ \linebreak on $\Config_{n,m}$ with $n, m$ determined by $\varphi$. Let $\mu := |M|$ be the size of the neighbourhood of $B$, and let $t$ be the size of $B$.

Then
\[
t \ge 2^{k^\epsilon}.
\]
In particular, for all sufficiently large $k$,
\[
\mu \ge k^\epsilon.
\]
\end{theorem}

\begin{proof}
    
In the previous Lemma~\ref{lm:V0_Q_sound} we proved that
\[
V^0 \vdash
\bigl[\mathrm{Refut}_{P_{\operatorname{CA}}}(x,y,s,B,\varphi)
\rightarrow
\varphi \notin SAT\bigr] .
\]
Thus $P_{\operatorname{CA}}$ satisfies the reflection principle from
Theorem~\ref{thm:reflection_theorem}. By Theorem~\ref{thm:reflection_theorem}, there exists a constant $d$
such that $F_d$ p-simulates $P_{\operatorname{CA}}$ for
refutations of CNF formulas.

We now apply this for $\varphi=\lnot \mathrm{ontoPHP}_k$ and $B$. It is a CNF formula. By the simulation, there exists an $F_d$-refutation of $\varphi$
whose size is at most $(t + |\varphi|)^{O(1)} = t^{O(1)}$ as $t = |B| \geq |\varphi|$.

However, by Ajtai's theorem (Theorem~\ref{thm:ajtai}), every $F_d$-refutation \linebreak
of $\lnot \mathrm{ontoPHP}_k$ has size at least
$2^{\,k^\delta}$
for some constant $\delta>0$.  It follows that 
$t \geq 2^{k^{\delta'}}$ for any $\delta' < \delta$ and $k >>0$, so every $P_{\operatorname{CA}}$-refutation of $\varphi$ has
exponential size.

Putting this into the previous inequality and using the upper bound on $t$ from Lemma~\ref{lm:CA_bound} and the bound on the set of states from Equation~\eqref{eq:set_of_states_bound}, we obtain
\[
2^{k^{\delta'}} \leq t \leq (c n m)^\mu \log(c n m) \leq  (c n m)^{\mu+1},
\]
for some constant $c>0$. Since $m=k(k+1)$ and $n=O(k^3)$, it follows that
\[
2^{k^{\delta'}} \le (d k^5)^{\mu +1},
\]
for some constant $d>0$.

This implies
\[
\mu \geq \Omega\!\left(\frac{k^{\delta'}}{\log k}\right).
\]

In particular, there exists $0<\epsilon<\delta'$ such that for all sufficiently large $k$
\[
\mu \ge k^{\epsilon} .
\]

As $\epsilon < \delta' < \delta$, 
\[
t \geq 2^{k^\epsilon}
\]
also holds and we are done.
\end{proof}

\section*{Open problems}

There are two open problems which may be interesting to explore:

\begin{enumerate}
\item Can the inverse implication in Theorem~\ref{thm:injective_not_sat} be formalized in $V^0$? 
We note that it can be formalized in $V^1$ (see~\cite{Cook_Nguyen_2010} for its definition). It remains open whether the whole argument can be carried out in a theory weaker than $V^1$.

    \item What is the strength of the proof system $P_\mathrm{CA}$? 
    In particular, does $P_\mathrm{CA}$ p-simulate resolution, or even bounded-depth Frege systems $F_d$ for some $d \geq 2$?
\end{enumerate}

\section*{Funding}

Maryia Kapytka was supported by Charles University Research Center program No. UNCE/24/SCI/022 and the project SVV-2025-260837.

\section*{Acknowledgements}

I would first like to express my gratitude to my advisor, Jan Krajíček, for his continuous support, guidance, and dedication. I am also deeply thankful to my friend and colleague, Ondřej Ježil, for patiently answering my many questions, reviewing parts of this paper, and providing invaluable emotional support.

I further wish to thank my friends and colleagues for their valuable assistance with grant applications and mathematical discussions, namely  Mykyta Narusevych, Daria Pavlova, Kateřina Panešová, Maroš Grego, Gabriel Krejčí, and Mavis Otrubová.

I acknowledge the use of ChatGPT (OpenAI) and Claude (Anthropic) to refine the academic language and improve the flow of this manuscript. The author reviewed and edited the content to ensure accuracy and takes full responsibility for the final work.

\bibliographystyle{plain}
\bibliography{bibliography}

\end{document}